\documentclass[reqno]{amsart}
\usepackage[a4paper,hmarginratio=1:1]{geometry}

\usepackage{amsmath}
\usepackage{mathabx}
\usepackage{amssymb}
\usepackage{mathtools}
\usepackage{amsthm}
\usepackage{graphicx}
\usepackage{titlesec}
\usepackage{lipsum}
\usepackage{comment}

\usepackage{enumitem}
\usepackage{tikz-cd}
\usepackage{dsfont}
\usepackage{lmodern}
\usepackage{bbm}
\usepackage{adjustbox}
\usepackage{hyperref}
\usepackage{mathrsfs}
\definecolor{dark-red}{rgb}{0.5,0.15,0.15}
\definecolor{dark-blue}{rgb}{0.15,0.15,0.6}
\definecolor{dark-green}{rgb}{0.15,0.6,0.15}
\hypersetup{
    colorlinks, linkcolor=dark-red,
    citecolor=dark-blue, urlcolor=dark-green
}
\usepackage[capitalize]{cleveref}
\newtheorem{theorem}{Theorem}[section]
\newtheorem{lemma}[theorem]{Lemma}
\newtheorem{proposition}[theorem]{Proposition}
\newtheorem{corollary}[theorem]{Corollary}
\theoremstyle{definition}
\newtheorem{definition}[theorem]{Definition}

\newtheorem{remark}[theorem]{Remark}

\numberwithin{equation}{section}

\newcommand{\loc}{\textup{Loc}}

\newcommand{\locm}{\textup{Loc}_{\otimes}}
\newcommand{\colocm}{\textup{Coloc}_{\textup{hom}}}

\newcommand{\lhf}{\textup{LH}\mathfrak{F}}

\newcommand{\tacstab}{\widetilde{\textup{Mod}}}

\newcommand{\tachom}{\widetilde{\textup{Hom}}}

\titleformat{\section}
  {\normalfont\normalsize\bfseries\filcenter}{\thesection.}{1em}{}
  \titleformat{\subsection}
  {\normalfont\normalsize\bfseries}{\thesubsection.}{1em}{}
\titleformat{\subsubsection}
  {\normalfont\normalsize\bfseries}{\thesubsubsection.}{1em}{}
  \title{The Balmer spectrum and telescope conjecture for infinite groups}
\author{Gregory Kendall}
\begin{document}
\begin{abstract}
    We determine the Balmer spectrum of dualisable objects in the stable module category for $\mathrm{H}_1\mathfrak{F}$ groups of type $\mathrm{FP}_{\infty}$ and show that the telescope conjecture holds for these categories. We also determine the spectrum of dualisable objects for certain infinite free products of finite groups. Using this, we give examples where the stable category is not stratified by the spectrum of dualisable objects and where the telescope conjecture does not hold. 
\end{abstract}
\maketitle
\section{Introduction}\let\thefootnote\relax\footnote{This work was supported by the Additional Funding Programme for Mathematical Sciences, delivered by EPSRC (EP/V521917/1) and the Heilbronn Institute for Mathematical Research.}
Given an essentially small tensor-triangulated category, Balmer \cite{Balmspec} introduced the spectrum, which is a topological space classifying thick tensor ideals of the category. Many such categories arise as the compact part of a larger rigidly-compactly generated tensor-triangulated category; the key example for this work is the stable module category of a finite group, although many other examples exist throughout mathematics; see \cite{balm} for a nice overview. 
\par 
On the other hand, there is less known about the tensor triangular geometry of non-rigidly-compactly generated tensor-triangulated categories, and there have only been a few explicit calculations of the Balmer spectra of non-rigid categories in the literature, for example \cite{BCSspec, XuSpec}. One reason for this, perhaps, is that non-rigidly-compactly generated categories can be quite different from each other, making a unifying theory more difficult.
\par 
Even within the modular representation theory of infinite groups we find that a tensor-triangulated category can fail to be rigidly-compactly generated in different ways: there are groups, as in \cite[Section~5.1]{kendall2024stablemodulecategorymodel}, where the tensor unit fails to be compact; there are groups, for example $\textup{H}_1\mathfrak{F}$ groups of type $\textup{FP}_{\infty}$, such that every dualisable object is compact but not vice versa; and there are groups, such as those which we consider in \Cref{freeprodsec}, such that every compact object is dualisable but there exist dualisable objects which are not compact. For this reason, we must tackle non-rigidly-compactly generated categories on a case by case basis. 
\par 
We investigate the tensor triangular geometry of the stable module category for various infinite groups. We constructed a suitable stable module category in \cite{kendall2024stablemodulecategorymodel} for all $\lhf$ groups, extending the construction of Mazza and Symonds in \cite{MSstabcat} for groups of type $\Phi$. Furthermore, in \cite{kendall2024costratificationstablecategoriesinfinite} we classified the localising tensor ideals, and colocalising hom-closed subcategories.
\par 
We will focus on the subcategory of dualisable objects; the main reason for this is that this forms an essentially small tensor-triangulated category, whereas the compact objects may not be closed under the tensor product. Furthermore, the results of \cite[Section~10]{barthel2023cosupport} provide evidence that for perfectly generated non-closed tensor-triangulated categories the spectrum of the dualisable objects is important to the tensor triangular geometry of the whole category. 
\par 
For certain groups, we determine the Balmer spectrum of dualisable objects as follows.
\begin{theorem}(See \Cref{bspec2} and \Cref{specfortrees}) 
Let $k$ be a field of characteristic $p > 0$ and $G$ be an $\textup{H}_1\mathfrak{F}$ group of type $\textup{FP}_{\infty}$. The Balmer spectrum of $\tacstab(kG)^d$ is homeomorphic to $\textup{Proj}(H^*(G,k))$.
\par 
Let $\mathcal{H}$ be a finite collection of finite groups of $p$-rank one, and for each $n \in \mathbb{N}$ let $H_n \in\mathcal{H}$. Suppose $G = \bigast\limits_{n\in\mathbb{N}}H_n$. Then $\textup{Spc}(\tacstab(kG)^d)$ is homeomorphic to the Stone-\v{C}ech compactification of $\mathbb{N}$. 
\end{theorem}
The statement for $\textup{H}_1\mathfrak{F}$ group of type $\textup{FP}_{\infty}$ extends the result of Benson, Carlson, and Rickard \cite{bcrthick} for finite groups. The situation when we have an infinite free product of cyclic groups of prime order extends the fact that the spectrum of an infinite product of fields is the Stone-\v{C}ech compactification of $\mathbb{N}$. Indeed, in view of \cite{ttrum}, we should think of the stable module category for a cyclic group of prime order as a tensor triangular field, and by work of G\'{o}mez \cite[Theorem~3.3]{gomez2024picardgroupstablemodule}, we know that the stable category splits as an infinite product in this case.
\par 
For more general free products $G = \bigast\limits_{n\in\mathbb{N}}H_n$, restriction induces a map from the spectral coproduct $\bigsqcup\limits_{n\in\mathbb{N}}\textup{Spc}(\tacstab(kH_n)^d) \to \textup{Spc}(\tacstab(kG)^d)$. We show in \Cref{NOSPECFORTREES} that if each $H_n$ has $p$-rank at least two, then this map will not be an epimorphism. In particular, the restriction maps to the finite elementary abelian subgroups form a collection of jointly conservative functors, which do not induce a jointly surjective map on the Balmer spectrum; comparing this to \cite{surjinttgeom} we see that non-rigidly-compactly generated categories can exhibit quite different behaviours to rigidly-compactly generated categories. 
\par 
In fact, as we note in \Cref{remarkprod}, this is really a general behaviour for infinite products of rigidly-compactly generated tensor-triangulated categories, under some additional conditions. 
\par 
In \cite[Section~10]{barthel2023cosupport}, it is shown that a theory of support and stratification may be defined for perfectly generated non-closed tensor-triangulated categories, extending the rigidly-compactly generated case. Using our calculation of the spectrum, and classification of localising tensor ideals from \cite{kendall2024costratificationstablecategoriesinfinite}, we give an example where the stable category is not stratified by the spectrum of dualisable objects, in the sense of \cite[Definition~10.30]{barthel2023cosupport}. 
\begin{theorem}(See \Cref{NOSTRAT})
    The stable category is not stratified by the Balmer spectrum of dualisable objects when $G$ is an infinite free product of cyclic groups. 
\end{theorem}
More generally, we show in \Cref{NOSPEC22} that if $\mathcal{T} = \prod\limits_{n\in\mathbb{N}}\mathcal{T}_n$, where each $\mathcal{T}_n$ is a tensor-triangular field, then $\mathcal{T}$ is not stratified by the Balmer spectrum of dualisable objects. 
\par 
We also consider smashing subcategories and the telescope conjecture for the stable categories, following Balmer and Favi \cite{BalmFavi}. The fact that the stable module category is not rigidly-compactly generated in general means we have to adapt the definition as follows. 
\par Recall that we say a localisation functor $L$ is smashing if $L$ is equivalent to $L(\mathbbm{1}) \otimes -$, where $\mathbbm{1}$ is the tensor unit. We note that for rigidly-compactly generated tensor-triangulated categories, this is equivalent to the statement that $L$ preserves coproducts, see \cite[Definition~3.3.2]{HPSasht}, although it is not clear if these are equivalent in general.
\par  
Smashing subcategories arise as kernels of smashing localisations. As we show in \Cref{tel2}, dualisable objects generate smashing tensor ideals, giving a well-defined map between the poset of thick tensor ideals of dualisable objects to the poset of smashing tensor ideals of the stable category. We say that the telescope conjecture holds if this map is a bijection. This extends Balmer and Favi's tensor-triangulated telescope conjecture to the not necessarily rigidly-compactly generated case.   
\par 
For finite groups, Benson, Iyengar, and Krause have shown that the telescope conjecture holds for the stable category \cite[Theorem~11.12]{repstrat}. For infinite groups, we show the following. 
\begin{theorem}(See \Cref{NOTELHERE}, \Cref{notelex} and \Cref{TELHERE}) 
    There exist $\textup{H}_1\mathfrak{F}$ groups such that the telescope conjecture does not hold for the stable category. However, if $G$ is $\textup{H}_1\mathfrak{F}$ and of type $\textup{FP}_{\infty}$, then the telescope conjecture does hold for the stable category.
\end{theorem}
The counterexamples to the telescope conjecture are given by infinite free products of groups; see \Cref{freeprodsec}.
\section{Background on tensor triangular geometry}
We begin by recalling all the relevant background. Let $\mathscr{T}$ be a tensor-triangulated category. We say a full subcategory $\mathcal{S}$ is thick if it is triangulated and closed under summands and finite sums; it is tensor ideal if it is closed under tensoring with any object in $\mathscr{T}$; and it is radical if $\mathcal{S} = \{X \in \mathscr{T}\ |\ X^{\otimes n}\in\mathscr{S}\textup{ for some }n \geq 1\}$. Note that when every object of $\mathscr{T}$ is dualisable then each thick tensor ideal is automatically radical, see \cite[Remark~4.3]{Balmspec}.
\par 
If $\mathscr{T}$ has infinite coproducts, we say a triangulated subcategory is localising if it is closed under infinite coproducts. 
\par 
For any object $X \in \mathscr{T}$, we use $\textup{Thick}_{\otimes}(X)$ to denote the smallest thick tensor ideal containing $X$, and similarly for $\locm(X)$. We will also use $\loc(X)$ to denote the smallest localising subcategory containing $X$, which will in general not be the smallest localising tensor ideal containing $X$. 
\par 
We also have the dual notions when there exists products and an internal hom. A subcategory $\mathcal{S}$ is colocalising if it is triangulated and closed under products, and it is hom-closed if for all $X \in \mathcal{S}$ we have $\textup{hom}(Y,X) \in \mathcal{S}$ for any $Y \in \mathscr{T}$, where $\textup{hom}(-,-)$ denotes the internal hom. The notation $\colocm(X)$ denotes the smallest colocalising hom-closed subcategory containing $X$. 
\par
An object $X$ is called rigid or strongly dualisable if the natural map $\textup{hom}(X,\mathbbm{1})\otimes Y \to \textup{hom}(X,Y)$ is an isomorphism for any $Y$. Since we will not consider any of form of dualisability, we will generally just refer to such objects as dualisable. 
\par 
We let $X^*=\textup{hom}(X,\mathbbm{1})$ be the dual of $X$, and note that if $X$ is dualisable then $-\otimes X$ is left adjoint to $X^*\otimes -$. 
\par
We also need to recall some facts about the Balmer spectrum. We approach this lattice theoretically as in \cite{KPLat} rather than the original exposition in \cite{Balmspec}. We begin with the relevant background on Stone duality. 
\par 
 A poset $(P, \leq)$ is called a lattice if all non-empty finite joins and non-empty finite meets exist; we will denote the join and meet by $\vee$ and $\wedge$ respectively. We will also assume that $P$ has a greatest and least element. A lattice is distributive if it satisfies $a \wedge (b \vee c) = (a \wedge b)\vee(a\wedge c)$ for all $a,b,c \in P$. 
 \par
 Furthermore, we say that $P$ is a frame if all (small) joins exist and the following distributivity law is satisfied for any set $I$ and any $x, \{y_i\}_{i \in I}$:
\begin{equation}
    x \wedge(\bigvee\limits_{i \in I}y_i) \leq \bigvee\limits_{i \in I}(x \wedge y_i)
\end{equation}
A frame homomorphism is then a poset homomorphism which preserves finite meets and arbitrary joins. 
\par 
When $P$ is a frame, we say an element $a \in P$ is finite if whenever $a \leq \bigvee\limits_{s \in S}s$ for any subset $S \subseteq P$, there exists a finite subset $K \subseteq S$ such that $a \leq \bigvee\limits_{k\in K}k$. We say $P$ is a coherent frame if every element can be expressed as a join of finite elements and the finite elements form a distributive lattice, which we will denote $P^f$. 
\par 
Spectral spaces are topological spaces which are homeomorphic to the spectrum of a commutative ring. These were described Hochster \cite{Hochster} as those spaces which are sober, quasi-compact, and such that the quasi-compact open subsets are closed under finite intersections and form a basis for the topology. A spectral map between spectral spaces is a continuous map such that the inverse image of a quasi-compact open subset is quasi-compact. 
\par 
We let $\mathbf{Spec}$ denote the category of spectral spaces with spectral maps, $\mathbf{DLat}$ denote the category of distributive lattices with lattice morphisms, and $\mathbf{CFrm}$ denote the category of coherent frames with frame homomorphisms. 
\begin{theorem}\label{Stonedual}(See, e.g., \cite[p.7]{BConlat}, \cite[Theorem~3.8]{barthel2023descenttensortriangulargeometry}) 
    The functor $\mathbf{CFrm} \to \mathbf{DLat}$ sending a coherent frame to its distributive lattice of finite elements is an equivalence.
    \par 
    The functor $\mathbf{Spec}^{op} \to \mathbf{CFrm}$ sending a spectral space to its coherent frame of open subsets is also an equivalence. 
\end{theorem}
In particular, there is a contravariant equivalence between the category of spectral spaces and the category of distributive lattices, given by sending a spectral space to its distributive lattice of quasi-compact open subsets \cite[Remark~3.9]{barthel2023descenttensortriangulargeometry}.
\begin{definition}
    We let $\textup{Thickid}(\mathscr{T})$ denote the poset of (radical) thick tensor ideals of $\mathscr{T}$ ordered by inclusion. 
    \par
    We let $\textup{Thickid}^f(\mathscr{T})$ be the distributive lattice of finite elements.
\end{definition}
By \cite[Theorem~3.1.9]{KPLat}, we know that the finite elements $\textup{Thickid}(\mathscr{T})$ are exactly those of the form $\textup{Thick}_{\otimes}(M)$ for some $M \in \mathscr{T}$.
\begin{definition}
    Assume $\mathscr{T}$ is an essentially small tensor-triangulated category. The Balmer spectrum of $\mathscr{T}$, denoted $\textup{Spc}(\mathscr{T})$, is the spectral space associated to $\textup{Thickid}^f(\mathscr{T})^{op}$ under Stone duality. 
\end{definition}
By \cite[Theorem~11.3]{barthel2023descenttensortriangulargeometry}, this agrees with Balmer's original definition in \cite{Balmspec}.
\par 
Given a tensor-triangulated functor $F: \mathscr{T} \to\mathscr{T}'$, we have an induced homomorphism of distributive lattices $\textup{Thickid}^f(\mathscr{T}) \to \textup{Thickid}^f(\mathscr{T}')$ defined by $\textup{Thick}_{\otimes}(M) \mapsto \textup{Thick}_{\otimes}(F(M))$, see \cite[Remark~11.11]{barthel2023descenttensortriangulargeometry}.
\par 
We record the following for later use. Recall that $f$ is an extremal epimorphism if whenever $f = g\circ h$ with $g$ a monomorphism, then $g$ is an isomorphism.
\begin{lemma}\label{reflectingthings}
Let $f:X \to Y$ be a morphism in $\mathbf{DLat}$. Then the following hold. 
\begin{enumerate}[label=(\roman*)]
    \item $f$ is a monomorphism if and only if $f$ is an injective homomorphism. 
    \item $f$ is an extremal epimorphism if and only if $f$ is a surjective homomorphism. 
    \item $f$ is an isomorphism if and only if $f$ is an injective and surjective homomorphism.
\end{enumerate}
\end{lemma}
\begin{proof}
    In \cite[Theorem~5.2.5]{Dickmann_Schwartz_Tressl_2019} it is shown that $f$ is injective if and only if its image in $\mathbf{Spec}$ under Stone duality is an epimorphism. The first statement follows. 
    \par 
    Similarly, \cite[Theorem~5.4.2]{Dickmann_Schwartz_Tressl_2019} shows that $f$ is surjective if and only if its image under Stone duality is an extremal monomorphism, and so $(ii)$ holds by duality. 
    \par 
    Now, any injective and surjective homomorphism is both a monomorphism and an extremal epimorphism by the previous two statements. It follows that it must be an isomorphism. On the other hand, suppose $f$ is an isomorphism. Then $f$ is a monomorphism, and so must be injective by $(i)$. It is easy to see that $f$ must be also be an extremal epimorphism and hence is surjective by $(ii)$. 
\end{proof}

We will also consider smashing subcategories, for which we need to recall some facts about localisations; see \cite{Kloc} for more details. A localisation functor is a pair $(L,\lambda)$ where $L:\mathscr{T} \to \mathscr{T}$ is an exact functor and $\lambda: \textup{Id}_{\mathscr{T}} \to L$ is a natural transformation such that $L\lambda =\lambda L$ and $L\lambda:L \to L^2$ is an isomorphism. 
\par 
A Bousfield subcategory $\mathcal{S}$ is a kernel of a localisation functor. Note that these are localising subcategories, and there are various equivalent formulations for a subcategory to be Bousfield, see \cite[Proposition~4.9.1]{Kloc}. In particular, $\mathcal{S}$ is Bousfield if and only if there exists a functorial triangle $\Delta_{\mathcal{S}}$ such that for each $t \in \mathscr{T}$ we have
\[\Gamma_{\mathcal{S}}(t) \to t \to L_{\mathcal{S}}(t) \to \Sigma\Gamma_{\mathcal{S}}(t)\]
such that $\Gamma_{\mathcal{S}}(t) \in \mathcal{S}$ and $L_{\mathcal{S}}(t) \in \mathcal{S}^{\perp}$. Note that, as usual, given a subcategory $\mathcal{S}$ we use $\mathcal{S}^{\perp}$ to denote the full subcategory of objects $X$ such that $\textup{Hom}_{\mathscr{T}}(S,X) = 0$ for all $S \in \mathcal{S}^{\perp}$.
\begin{proposition}\label{smashingdefinition}
    Suppose $\mathcal{S}$ is a Bousfield localising tensor ideal. The following are equivalent. 
    \begin{enumerate}[label=(\roman*)]
        \item $\mathcal{S}^{\perp}$ is a localising tensor ideal
        \item There is an isomorphism of functors $L_{\mathcal{S}}(-) \simeq L_{\mathcal{S}}(\mathbbm{1}) \otimes -$
        \item There is an isomorphism of functors $\Gamma_{\mathcal{S}}(-) \simeq \Gamma_{\mathcal{S}}(\mathbbm{1}) \otimes -$
        \item There is an isomorphism of functors $\Delta_{\mathcal{S}}(-) \simeq \Delta_{\mathcal{S}}(\mathbbm{1}) \otimes -$
    \end{enumerate}
\end{proposition}
\begin{proof}
    This is almost \cite[Proposition~2.5]{HallRydh}, which shows the equivalence of $(ii),(iii)$ and $(iv)$ with the statement that $\mathcal{S}^{\perp}$ is tensor ideal. The only part we need to check is that if, say, $(ii)$ holds then $\mathcal{S}^{\perp}$ is localising. However, this is immediate since tensor products commute with coproducts.
\end{proof}
\begin{definition}
We say a localising tensor ideal $\mathcal{S}$ is smashing if it satisfies the equivalent conditions of \Cref{smashingdefinition}.
\end{definition} 
\begin{remark}
    If $\mathscr{T}$ is rigidly compactly generated this is equivalent to the a priori weaker statement that $\mathcal{S}^{\perp}$ is closed under coproducts as shown in \cite[Proposition~2.5]{HallRydh}; this property is often used to define smashing ideals, but it is not clear if it is equivalent in general. 
\end{remark}
Finally, we briefly recall the telescope conjecture for rigidly-compactly generated tensor-triangulated categories following Balmer and Favi \cite{BalmFavi}. Let $\mathscr{T}$ be a rigidly-compactly generated tensor triangulated category. Then there is an injective map $\textup{Thickid}(\mathscr{T}^c) \to \textup{Smashid}(\mathscr{T})$ defined by $\mathcal{T} \mapsto \locm(\mathcal{T})$, see \cite[Theorem~4.1]{BalmFavi} or \cite[Corollary~2.1]{HallRydh} for example.
\begin{definition}\cite[Definition~4.2]{BalmFavi} 
    We say that the (rigidly-compactly generated) telescope conjecture holds for $\mathscr{T}$ if the natural map $\textup{Thickid}(\mathscr{T}^c) \to \textup{Smashid}(\mathscr{T})$ is bijective. 
\end{definition}
In view of the fact this map is always injective for rigidly-compactly generated categories, the hard part is showing that it is surjective, that is, in showing that smashing ideals are generated by compact objects. 
\section{The stable module category}
We also need to recall some facts about the stable module category $\tacstab(kG)$ for $\textup{LH}\mathfrak{F}$ groups from \cite{kendall2024costratificationstablecategoriesinfinite,kendall2024stablemodulecategorymodel}. Throughout, we will take $k$ to be a field of characteristic $p > 0$.
\par 
For the definition of $\textup{LH}\mathfrak{F}$ groups in full generality, see \cite{krop}. We will only really be interested in $\textup{H}_1\mathfrak{F}$ groups, which are groups that admits a cellular action on a finite dimensional contractible cell complex such that each stabiliser is finite. For example, this includes all groups with a finite dimensional model for the classifying space for proper actions, and in particular all groups of finite virtual cohomological dimension. 
\par 
We refer to \cite{kendall2024stablemodulecategorymodel} for the general construction of the stable category, and only recall here the properties we will need. We will use $\tachom_{kG}(-,-)$ to denote the morphisms in the stable category.
\par 
As shown in \cite[Section~4]{kendall2024stablemodulecategorymodel}, the stable category is a closed, tensor-triangulated category, where the structure is induced by the usual tensor product over $k$ with diagonal $G$-action, which we will usually denote with an unadorned tensor product $-\otimes -$. 
\begin{lemma}
    Restriction, induction, and coinduction all induce functors on the stable category. Restriction is right adjoint to induction, and restriction is left adjoint to coinduction.
\end{lemma}
\begin{proof}
    This is shown in \cite[Theorem~5.6]{kendall2024costratificationstablecategoriesinfinite}. 
\end{proof}
\begin{definition}
    Suppose $F \leq G$ is a subgroup, $M$ is a $kF$-module, and $N$ is a $kG$-module. We will write $M{\uparrow}_F^G$ and $M{\Uparrow}_F^G$ to denote induction and coinduction respectively. Similarly, $N{\downarrow}_F^G$ will denote restriction.
\end{definition}
\begin{proposition}\label{comkin}
    Let $G$ be an $\lhf$ group and $k$ a field of characteristic $p > 0$. Then $\tacstab(kG)$ is compactly generated, with a set of compact generators being all modules of the form $k{\uparrow}_E^G$, with $E \leq G$ a finite elementary abelian subgroup. 
\end{proposition}
\begin{proof}
    First, we note that for elementary abelian groups $E$, the stable category is compactly generated by the trivial module $k$, since this is the only simple module. 
    \par 
    \sloppy If $M$ is a $kG$-module such that $\tachom_{kG}(k{\uparrow}_E^G,M) = 0$, then we see by adjunction that $\tachom_{kG}(k,M{\downarrow}_E^G) = 0$. Hence, $M{\downarrow}_E^G \cong 0$ for each finite elementary abelian subgroup $E \leq G$. By Chouinard's theorem we then conclude that $M{\downarrow}_F^G \cong 0$ for every finite subgroup $F \leq G$. It now follows from \cite[Theorem~2.30]{kendall2024stablemodulecategorymodel} that $M \cong 0$. 
\end{proof}
Note that for the above to hold, we really only need to consider maximal elementary abelian $p$-subgroups. 
\par 
We will need the following lemma later. 
\begin{lemma}\label{isol}
    Suppose $M$ and $N$ are $kG$-modules such that $M \in \locm(N)$ and $E \leq G$ is a subgroup. Then $M{\downarrow}_E^G \in \locm(N{\downarrow}_E^G)$.
    \par 
    Similarly, if $X$ and $Y$ are $kE$-modules such that $X \in \locm(Y)$, then $X{\uparrow}_E^G \in \locm(Y{\uparrow}_E^G)$.
\end{lemma}
\begin{proof}
    Consider the full subcategory of $\tacstab(kG)$ of modules $X$ such that $X{\downarrow}_E^G \in \locm(X{\downarrow}_E^G)$. This is easily checked to be a localising tensor ideal and hence it contains $\locm(N)$. It follows that $M{\downarrow}_E^G\in \locm(X{\downarrow}_E^G)$.
    \par 
    The second statement follows in a similar way, although we use \cite[Proposition~3.5]{kendall2024costratificationstablecategoriesinfinite} to see that the full subcategory of $\tacstab(kE)$ of modules $Z$ such that $Z{\uparrow}_E^G\in\locm(Y{\uparrow}_E^G)$ is a localising tensor ideal. 
\end{proof}
In \cite{kendall2024costratificationstablecategoriesinfinite} we classified the localising tensor ideals and colocalising hom-closed subcategories. We explain why we can reduce to the elementary abelian subgroups, rather than all finite subgroups. We let $\mathcal{E}(G)$ denote the collection of finite elementary abelian $p$-subgroups of $G$. 
\par 
Recall the following Quillen category. 
\begin{definition}\label{aegcatdef}
    The category $\mathcal{A}_{\mathcal{E}}(G)$ is defined as follows: the objects are all elementary abelian $p$-subgroups of $G$ and the morphisms are induced by inclusions and conjugations. 
\end{definition}
We also recall some notation from \cite{kendall2024costratificationstablecategoriesinfinite}. For a group $G$, we let $\loc(G)$ be the set of localising tensor ideal subcategories of $\tacstab(kG)$. As shown in \cite[Section~3]{kendall2024costratificationstablecategoriesinfinite} this induces a functor $\mathcal{A}_{\mathcal{E}}(G)^{op} \to \textup{Set}$, see in particular the discussion after \cite[Lemma~3.17]{kendall2024costratificationstablecategoriesinfinite}. A subgroup inclusion $E \leq F$ induces a map $\loc(F) \to \loc(E)$ given by $\mathcal{L}\mapsto \locm(\mathcal{L}{\downarrow}_E^F) := \locm(\{X{\downarrow}_E^F\ |\ X \in \mathcal{L}\})$, and conjugations induce the obvious maps. 
\par 
There is a natural map $\loc(G) \to \lim\limits_{E \in \mathcal{A}_{\mathcal{E}}(G)}\loc(E)$ induced by restriction. 
\begin{theorem}\label{corlomain}
    The natural map $\loc(G) \to \lim\limits_{E \in \mathcal{A}_{\mathcal{E}}(G)}\loc(E)$ is invertible. 
\end{theorem}
\begin{proof}
    As shown in the proof of \Cref{comkin}, Chouinard's theorem allows us to show that for a $kG$-module $M$, we have $M\cong 0$ if and only if $M{\downarrow}_E^G \cong 0$ for every finite elementary abelian subgroup $E \leq G$. In the language of \cite{kendall2024costratificationstablecategoriesinfinite}, this means that the detection property holds for the collection of elementary abelian subgroups of $G$. Given this, the result now follows from \cite[Theorem~3.24]{kendall2024costratificationstablecategoriesinfinite}.
\end{proof}
We extract the following in the form we will require later. 
\begin{corollary}\label{hbil}
    Let $M$ and $N$ be $kG$-modules. Then $M \in \locm(N)$ if and only if $M{\downarrow}_E^G \in \locm(N{\downarrow}_E^G)$ for every finite elementary abelian subgroup $E \leq G$. 
    \par 
    Similarly, $M \in \colocm(N)$ if and only if $M{\downarrow}_E^G \in \colocm(N{\downarrow}_E^G)$ for each finite elementary abelian subgroup $E \leq G$.
\end{corollary}
\begin{proof}
    We show only the claim about localising tensor ideals, with the claim about colocalising hom-closed subcategories being dual. We have already noted in \Cref{isol} that if $M \in \locm(N)$ then $M{\downarrow}_E^G \in \locm(N{\downarrow}_E^G)$ for every finite elementary abelian subgroup $E \leq G$. 
    \par 
    On the other hand, suppose $M{\downarrow}_E^G \in \locm(N{\downarrow}_E^G)$ for every finite elementary abelian subgroup $E \leq G$. Then we immediately conclude by \Cref{corlomain} that $M \in \locm(N)$. 
\end{proof}
We also need the following result, which is essentially from \cite{kendall2024costratificationstablecategoriesinfinite}, although we again explain why we can reduce to elementary abelian subgroups.
\begin{proposition}\label{detectedelsub}
Let $M$ be a $kG$-module. Then $\locm(M) = \locm(\bigoplus\limits_{E \in \mathcal{E}(G)} M{\downarrow}_E^G{\uparrow}_E^G)$ and $\colocm(M) = \colocm(\prod\limits_{E\in\mathcal{E}(G)} M{\downarrow}_E^G{\Uparrow}_E^G)$. 
\end{proposition}
\begin{proof}
    As we mentioned in the proof of \Cref{corlomain}, in the language of \cite{kendall2024costratificationstablecategoriesinfinite}, we have that the detection property holds for the collection of finite elementary abelian $p$-subgroups of $G$. The statement is now shown in \cite[Theorem~3.24]{kendall2024costratificationstablecategoriesinfinite} and \cite[Theorem~4.10]{kendall2024costratificationstablecategoriesinfinite}.
\end{proof}
We also introduce some notation for ease. We let $\tacstab(kG)^c$ and $\tacstab(kG)^d$ denote the subcategories of compact objects and dualisable objects respectively. We will also write $\textup{Thickid}(G^d)$ and $\textup{Thickid}^f(G^d)$ to denote $\textup{Thickid}(\tacstab(kG)^d)$ and $\textup{Thickid}^f(\tacstab(kG)^d)$ respectively. 
\subsection{Support}\label{subsectionofsupport}
We briefly recall the support for infinite groups as defined by Benson \cite[Section~15]{bensoninf}. As we showed in \cite[Proposition~5.11]{kendall2024costratificationstablecategoriesinfinite} this support is exactly the support classifying the localising tensor ideals of the stable category. 
\par 
For a finite elementary abelian group $F$, we let $\textup{Supp}_F$ denote the usual support as in \cite{BCRvars,BIKloc}. For each finite elementary abelian subgroup $F \leq G$, we have a natural map $\rho^*_F: \textup{Proj}(H^*(F,k)) \to \underset{{E \in \mathcal{A}_{\mathcal{E}}(G)}}{\textup{colim}}\textup{Proj}(H^*(E,k)) $.
\begin{definition}
    Let $G$ be an $\textup{LH}\mathfrak{F}$ group and $M$ be a $kG$-module. We define the support of $M$, denoted $\textup{Supp}_G(M)$, as the union over all finite elementary abelian subgroups $E \leq G$ of $\rho_E^*(\textup{Supp}_E(M{\downarrow}_E^G))$. 
\end{definition}
By the definition, we find that many of the properties of support for finite groups will carry over. 
\begin{proposition}\label{prop:supportprop}
The support has the following properties. 
\begin{enumerate}[label=(\roman*)]
\item For any $kG$-module $M$, $\textup{Supp}_G(M) = \emptyset \iff M \cong 0$ in $\tacstab(kG)$
\item For any $kG$-module $M$, $\textup{Supp}_G(\Omega^{-1}M) = \textup{Supp}_G(M)$
\item For any exact triangle $M_1 \to M_2 \to M_3 \to \Omega^{-1}(M_1)$ in $\tacstab(kG)$, we have $\textup{Supp}_G(M_2) \subseteq \textup{Supp}_G(M_1) \cup \textup{Supp}_G(M_3)$
\item For any collection $\{M_{\alpha}\}$ of $kG$-modules, $\textup{Supp}_G(\bigoplus\limits_{\alpha}M_{\alpha}) = \bigcup\limits_{\alpha}\textup{Supp}_G(M_{\alpha})$
\item For any $kG$-modules $M$ and $N$, $\textup{Supp}_G(M\otimes N) = \textup{Supp}_G(M) \cap \textup{Supp}_G(N)$
\end{enumerate}
\end{proposition}
\begin{proof}
\begin{enumerate}[label=(\roman*)]
\item This holds since $M \cong 0$ in $\tacstab(kG)$ if and only if $M{\downarrow}_F^G \cong 0$ in $\tacstab(kF)$ for each finite subgroup $F \leq G$. By Chouinard's theorem, this is true if and only if $M{\downarrow}_E^G \cong 0$ in $\tacstab(kE)$ for each finite elementary abelian subgroup. 
\item This follows since restriction commutes with syzygies and the corresponding statement is true over every finite group by \cite[Proposition~5.1]{BIKloc}.
\item The corresponding statement is true for finite groups by \cite[Proposition~5.1]{BIKloc} and so it holds for $G$ by definition. 
\item Restriction commutes with direct sums and so this follows by restricting to finite subgroups, where the corresponding result holds by \cite[Corollary~6.6]{BIKloc}. 
\item Restricting to any elementary abelian subgroup commutes with tensor products. The tensor product formula holds over finite groups \cite[Theorem~11.1]{repstrat} and so it holds for $G$ by definition.
\end{enumerate}
\end{proof}
We can also determine the support of induced modules. 
\begin{lemma}\label{inducedsupport}
    Let $E \leq G$ be an elementary abelian subgroup, and $M$ a $kE$-module. Then $\textup{Supp}_G(M{\uparrow}_E^G) = \rho_E^*(\textup{Supp}_E(M))$.
\end{lemma}
\begin{proof}
    This follows by the Mackey formula. 
\end{proof}
This support classifies the localising tensor ideals as follows. Note that in \cite{kendall2024costratificationstablecategoriesinfinite} we worked over all finite subgroups, but as we have shown above we can in fact reduce to the elementary abelian subgroups, since $k$ is a field of characteristic $p > 0$.
\begin{proposition}\label{newsix}
    Let $M$ and $N$ be $kG$-modules. Then $\locm(M) = \locm(N)$ if and only if $\textup{Supp}_G(M) = \textup{Supp}_G(N)$. 
\end{proposition}
\begin{proof}
    It is shown in \cite[Proposition~5.11]{kendall2024costratificationstablecategoriesinfinite} that the map sending a localising tensor ideal $\mathcal{L}$ to $\textup{Supp}_G(\mathcal{L})$ is a bijection. The result follows.
\end{proof}
Note that it is possible to define a cosupport in a completely analogous manner, which would then classify the colocalising hom-closed subcategories. 
\section{Smashing subcategories}
In this section, we consider smashing tensor ideals and the telescope conjecture for the stable category. We continue to assume that $G$ is an $\lhf$ group and $k$ is a field of characteristic $p > 0$. 
\par 
Note that it is not true in general that localising subcategories are automatically Bousfield subcategories; however we can show this is true for the stable category.
\begin{proposition}\label{genbyobj}
    Suppose $\mathcal{L}$ is a localising tensor ideal of $\tacstab(kG)$. Then $\mathcal{L}$ is generated by a set of objects.
\end{proposition}
\begin{proof}
    Let $E \leq G$ be a finite elementary abelian subgroup. We know from \cite[Lemma~11.11]{repstrat} that $\locm(\mathcal{L}{\downarrow}_E^G)$ is a Bousfield subcategory, noting that the authors of \cite{repstrat} use the term strictly localising subcategory instead of Bousfield subcategory. Hence $\locm(\mathcal{L}{\downarrow}_E^G)$ is generated by a set of objects by \cite[Lemma~3.3.1]{KSprojline}. We let $\mathcal{C}_E$ be a set of generators for $\locm(\mathcal{L}{\downarrow}_E^G)$ for each finite elementary abelian subgroup $E \leq G$. We claim that $\mathcal{L}$ is generated by all modules of the form $M{\uparrow}_E^G$ for $M \in \mathcal{C}_E$ for some finite elementary abelian subgroup $E \leq G$, which shows that it is generated by a set. 
    \par 
    Firstly, recall from \cite[Lemma~3.6]{kendall2024costratificationstablecategoriesinfinite} that $\locm(\mathcal{L}{\downarrow}_E^G)$ coincides with the collection of all $kE$-modules $X$ such that $X{\uparrow}_E^G \in \mathcal{L}$. It follows that $M{\uparrow}_E^G \in \mathcal{L}$ and hence the localising tensor ideal generated by all such $M{\uparrow}_E^G$ is contained in $\mathcal{L}$. We are therefore left to show the other inclusion. Suppose that $X \in \mathcal{L}$. By definition, we know that $X{\downarrow}_E^G \in \locm(\mathcal{C}_E)$. Hence, from \Cref{isol} we know that $X{\downarrow}_E^G{\uparrow}_E^G \in \locm(\mathcal{C}_E{\uparrow}_E^G)$. This is true for all finite elementary abelian $E \leq G$ and so by \Cref{detectedelsub} we conclude that $X$ is in the subcategory generated by all modules of the form $M{\uparrow}_E^G$ for $M \in \mathcal{C}_E$ for some finite elementary abelian subgroup $E \leq G$. 
\end{proof}
\begin{corollary}
    Every localising tensor ideal of $\tacstab(kG)$ is a Bousfield subcategory.
\end{corollary}
\begin{proof}
    Given \Cref{genbyobj}, this is essentially a consequence of Brown representability. Since every localising tensor ideal $\mathcal{L}$ is generated by a set of objects, and $\tacstab(kG)$ is compactly generated, it follows from \cite[Theorem~7.2.1]{Kloc} that $\mathcal{L}$ is well generated. The result now follows from \cite[Theorem~5.2.1]{Kloc}. 
\end{proof}
This means that given any localising tensor ideal subcategory $\mathcal{S}$, we can consider the associated Bousfield localisation functor. We can therefore ask whether or not a localising tensor ideal is smashing. 
\begin{definition}
    We let $\textup{Smashid}(G)$ denote the poset of smashing ideals of $\tacstab(kG)$.  
\end{definition}
We now move on to showing that smashing ideals are generated by induced modules. We first isolate a simple argument that we will use a few times. We use the notation $\loc(M \otimes \tacstab(kG))$ as shorthand for $\loc(\{M \otimes X\ |\ X \in \tacstab(kG)\})$.
\begin{lemma}\label{lop}
    Let $M$ be a $kG$-module. Then for any $kG$-module $N$ we have $N \in \loc_{\otimes}(M)^{\perp}$ if and only if $\tachom_{kG}(M\otimes X,N) = 0$ for all $kG$-modules $X$.
    \par 
    Furthermore, suppose we have a collection $\{M_i\}_{i \in I}$ of $kG$-modules. Then $N \in \locm(\{M_i\}_{i \in I})^{\perp}$ if and only if $\tachom_{kG}(M_i \otimes X,N) = 0$ for all $kG$-modules $X$ and $i \in I$.
\end{lemma}
\begin{proof}
    The forward implication is obvious, so assume that $\tachom_{kG}(M\otimes X,N) = 0$ for all $kG$-modules $X$. First, we note that by \cite[Lemma~2.1]{BENSON2011953} we know that $\locm(M) = \loc(M \otimes \tacstab(kG))$. 
    \par 
    Consider the full subcategory of $kG$-modules $Y$ such that $\tachom_{kG}(Y,N) = 0$. This is a localising subcategory which, by assumption, contains $M \otimes X$ for each $kG$-module $X$. Therefore, it contains $\loc(M \otimes \tacstab(kG))$. It follows that $N \in \loc(M \otimes \tacstab(kG))^{\perp} = \locm(M)^{\perp}$ as desired. 
    \par 
    The second statement follows in the same manner.
\end{proof}
\begin{lemma}\label{hop}
    Suppose $\mathcal{L}$ is a localising tensor ideal generated by modules of the form $N{\uparrow}_E^G$, where $E \leq G$ is a finite elementary abelian subgroup and $N$ is a finitely generated $kE$-module. Then $\mathcal{L}$ is smashing.
\end{lemma}
\begin{proof}
    We begin by assuming we have a localising tensor ideal of the form $\locm(N{\uparrow}_E^G)$, where $E \leq G$ is a finite elementary abelian subgroup and $N$ is a finitely generated $kE$-module. Suppose that we have $X \in \locm(N{\uparrow}_E^G)^{\perp}$ and $Y$ is any $kG$-module. 
    \par 
    We claim that $X \in \locm(N{\uparrow}_E^G)^{\perp}$ if and only if $N^* \otimes X{\downarrow}_E^G \cong 0$. From this it follows immediately that $\locm(N{\uparrow}_E^G)^{\perp}$ is tensor ideal, and hence $\locm(N{\uparrow}_E^G)$ is smashing. 
    \par 
    For the claim, we have the following implications. 
    \begin{align}
        X \in \locm(N{\uparrow}_E^G)^{\perp} &\iff \tachom_{kG}(Z \otimes N{\uparrow}_E^G,X) = 0\textup{ for all }kG\textup{-modules }Z\\ &\iff\tachom_{kE}(Z{\downarrow}_E^G\otimes N, X{\downarrow}_E^G) = 0 \textup{ for all }kG\textup{-modules }Z\\ &\iff \tachom_{kE}(Z{\downarrow}_E^G,N^*\otimes X{\downarrow}_E^G) = 0 \textup{ for all }kG\textup{-modules }Z\\ &\iff N^*\otimes X{\downarrow}_E^G \cong 0
    \end{align}
    The first equivalence follows from \Cref{lop}. The second equivalence uses the isomorphism $Z\otimes N{\uparrow}_E^G \cong (Z{\downarrow}_E^G\otimes N){\uparrow}_E^G$ along with restriction-induction adjunction. The third equivalence uses that $N$ is dualisable in $\tacstab(kE)$ and the final equivalence follows since we may choose $Z = k$, and $k$ is a compact generator in the stable category of an elementary abelian $p$-group. 
    \par 
    The case where $\mathcal{L}$ is generated by a set of modules of the form $N{\uparrow}_E^G$ follows by a similar argument.
\end{proof}
Before we move on, we introduce the following notation. 
\begin{definition}
    Let $F \leq G$ be a subgroup and suppose $\mathcal{T}$ is a full subcategory of $\tacstab(kG)$. We let $\textup{Ind}^{-1}_{G,F}(\mathcal{T})$ denote the full subcategory of $\tacstab(kF)$ consisting of all $kF$-modules $M$ such that $M{\uparrow}_F^G \in \mathcal{T}$.
    \par 
    We define $\textup{Coind}^{-1}_{G,F}(\mathcal{T})$ in a similar way.
\end{definition}
Note that if $\mathcal{T}$ is a localising tensor ideal, then $\textup{Ind}^{-1}_{G,F}(\mathcal{T})$ is also a localising tensor ideal, as shown in \cite[Proposition~3.5]{kendall2024costratificationstablecategoriesinfinite}. Similarly, if $\mathcal{T}$ is a colocalising hom closed subcategory, then $\textup{Coind}^{-1}_{G,F}(\mathcal{T})$ is also a colocalising hom closed subcategory, which is shown in \cite[Proposition~4.2]{kendall2024costratificationstablecategoriesinfinite}.
\begin{lemma}\label{bsl}
    Suppose $\mathcal{T}$ is a localising tensor ideal and a colocalising hom-closed subcategory. Then $\textup{Ind}^{-1}_{G,E}(\mathcal{T}) = \textup{Coind}^{-1}_{G,E}(\mathcal{T})$ for each finite elementary abelian subgroup $E \leq G$.
\end{lemma}
\begin{proof}
    We showed in \Cref{hbil} that for any $kG$-modules $X$ and $Y$ we have $X \in \colocm(Y)$ if and only if $X{\downarrow}_E^G \in \colocm(Y{\downarrow}_E^G)$ for each finite elementary abelian subgroup $E \leq G$. 
    \par 
    Let $M$ be a $kE$-module such that $M{\uparrow}_E^G \in \mathcal{T}$. By the Mackey formula, we know that $M{\Uparrow}_E^G{\downarrow}_F^G \in \colocm(M{\uparrow}_E^G{\downarrow}_F^G)$ for every finite elementary abelian subgroup $F \leq G$. Therefore, we know that $M{\Uparrow}_E^G \in \colocm(M{\uparrow}_E^G)$. Since $\mathcal{T}$ is a colocalising hom-closed subcategory containing $M{\uparrow}_E^G$ by assumption, we know that $\colocm(M{\uparrow}_E^G) \subseteq \mathcal{T}$. It follows that $M \in \textup{Coind}^{-1}_{G,E}(\mathcal{T})$ and so $\textup{Ind}^{-1}_{G,E}(\mathcal{T}) \subseteq \textup{Coind}^{-1}_{G,E}(\mathcal{T})$.
    \par 
    A similar argument shows the inclusion $\textup{Coind}^{-1}_{G,E}(\mathcal{T}) \subseteq \textup{Ind}^{-1}_{G,E}(\mathcal{T})$ and hence we find the claimed equality. 
\end{proof}
\begin{lemma}\label{lemmafromoth}
    Suppose $\mathcal{S}$ is a localising tensor ideal and $E \leq G$ a finite elementary abelian subgroup. Then $\textup{Coind}_{G,E}^{-1}(\mathcal{S}^{\perp}) = (\textup{Ind}^{-1}_{G,E}(\mathcal{S}))^{\perp}$.
\end{lemma}
\begin{proof}
    The dual result is shown during the proof of \cite[Proposition~4.11]{kendall2024costratificationstablecategoriesinfinite}. We include the details here for completeness. Suppose first that we have some $kE$-module $X$ such that $X{\Uparrow}_E^G \in \mathcal{S}^{\perp}$ and let $Y$ be such that $Y{\uparrow}_E^G \in \mathcal{S}$. 
    \par 
    \sloppy We need to show that $\tachom_{kE}(Y,X) = 0$. First, note that by assumption we have $\tachom_{kG}(Y{\uparrow}_E^G,X{\Uparrow}_E^G) = 0$. It follows by adjunction that $\tachom_{kE}(Y{\uparrow}_E^G{\downarrow}_E^G,X) = 0$ and since $Y$ is a summand of $Y{\uparrow}_E^G{\downarrow}_E^G$, we conclude that $\tachom_{kE}(Y,X) = 0$ as desired. 
    \par 
    On the other hand, suppose $M \in (\textup{Ind}^{-1}_{G,E}(\mathcal{S}))^{\perp}$ and that $N \in \mathcal{S}$. We need to show that $\tachom_{kG}(N,M{\Uparrow}_E^G) = 0$. Note that $N{\downarrow}_E^G{\uparrow}_E^G \in \mathcal{S}$, since $\mathcal{S}$ is tensor ideal, and so by assumption $\tachom_{kG}(N{\downarrow}_E^G,M) = 0$. By adjunction, we now conclude immediately that $\tachom_{kG}(N,M{\Uparrow}_E^G) = 0$ as desired.
\end{proof}
\begin{lemma}\label{smpr}
    Suppose that $\mathcal{S}$ is a smashing subcategory. Then $\textup{Ind}^{-1}_{G,E}(\mathcal{S})$ is also smashing.
\end{lemma}
\begin{proof}
    The key point to observe here is that $\mathcal{S}^{\perp}$ is a localising tensor ideal and a colocalising hom-closed subcategory. We know that $\textup{Ind}^{-1}_{G,E}(\mathcal{S}^{\perp}) = \textup{Coind}^{-1}_{G,E}(\mathcal{S}^{\perp}) = (\textup{Ind}^{-1}_{G,E}(\mathcal{S}))^{\perp} $; the first equality is from \Cref{bsl} and the second is \Cref{lemmafromoth}. It follows that $(\textup{Ind}^{-1}_{G,E}(\mathcal{S}))^{\perp}$ is a localising tensor ideal and hence $\textup{Ind}^{-1}_{G,E}(\mathcal{S})$ is smashing. 
\end{proof}
\begin{proposition}\label{smagenco}
    Every smashing ideal is generated by compact objects of the form $M{\uparrow}_E^G$ where $E \leq G$ is a finite elementary abelian subgroup and $M$ is a finitely generated $kE$-module. 
\end{proposition}
\begin{proof}
    Let $\mathcal{S}$ be smashing. Then \Cref{smpr} shows that $\textup{Ind}^{-1}_{G,E}(\mathcal{S})$ is smashing. The telescope conjecture for finite groups \cite[Theorem~11.12]{repstrat} shows that $\textup{Ind}^{-1}_{G,E}(\mathcal{S})$ is generated by a set of compact objects, say $\mathcal{C}_E$.
    \par 
    We claim that $\mathcal{S}$ is generated as a localising tensor ideal by all objects of the form $X{\uparrow}_E^G$ where $E \leq G$ is a finite elementary abelian subgroup and $X \in \mathcal{C}_E$; we let $\mathcal{J}$ be the localising ideal generated by all such $X{\uparrow}_E^G$. Clearly, $X{\uparrow}_E^G \in \mathcal{S}$ and so $\mathcal{J}\subseteq \mathcal{S}$.
    \par 
    On the other hand, suppose that $Y \in \mathcal{S}$. Then for any finite elementary abelian subgroup $E \leq G$ we have $Y{\downarrow}_E^G\in \locm(\mathcal{C}_E)$. We conclude from \Cref{isol} that $Y{\downarrow}_E^G{\uparrow}_E^G \in \locm(\bigcup\limits_{X \in \mathcal{C}_E}X{\uparrow}_E^G)$. This is true for each finite elementary abelian subgroup $E \leq G$ and so by \Cref{detectedelsub} we see that $Y \in \mathcal{J}$. 
\end{proof}
\begin{theorem}
    A localising tensor ideal is smashing if and only if it is generated by a collection of modules of the form $N{\uparrow}_E^G$ where $E \leq G$ is a finite elementary abelian subgroup and $N$ is a finitely generated $kE$-module.
\end{theorem}
\begin{proof}
    This follows from \Cref{hop} and \Cref{smagenco}. 
\end{proof}
Note that we are not claiming that a localising tensor ideal is smashing if and only if it is generated by compact objects. Indeed, it is not clear if a localising tensor ideal generated by an arbitrary compact objects will be smashing in general.
\subsection{The telescope conjecture}
We now consider the appropriate version of the telescope conjecture in our situation. The compact objects will have no reason to be closed under the tensor product in the stable category. Hence, it does not even make sense to consider thick tensor ideals of compact objects. Instead, we must focus on the dualisable objects, which do form a tensor-triangulated category. We begin by showing that dualisable objects generate smashing subcategories. For this, we need the following well known observation. 
\begin{lemma}\label{du}
    Let $\mathcal{T}$ be a thick ideal of dualisable objects. Then $X \in \mathcal{T}$ if and only if $X^* \in \mathcal{T}$.
\end{lemma}
\begin{proof}
    This follows immediately from the fact that if $X$ is dualisable then $X$ is a summand of $X \otimes X^*\otimes X$ \cite[Lemma~A.2.6]{HPSasht}, and the natural map $X \to X^{**}$ is an isomorphism \cite[Theorem~A.2.5]{HPSasht}.
\end{proof}
\begin{proposition}\label{tel2}
    Suppose $\mathcal{T}$ is a thick tensor ideal of $\tacstab(kG)^d$. Then $\locm(\mathcal{T})$ is a smashing tensor ideal. 
\end{proposition}
\begin{proof}
\sloppy We claim that $\locm(\mathcal{T})^{\perp} = \bigcap\limits_{T \in \mathcal{T}}\textup{ker}(T\otimes -)$. From this it is clear that $\locm(\mathcal{T})^{\perp}$ is tensor ideal and so $\locm(\mathcal{T})$ must be smashing. 
\par 
We now have the following chain of implications. 
\begin{align}
    X \in \locm(\mathcal{T})^{\perp} &\iff \tachom_{kG}(T \otimes Y,X) = 0 \textup{ for all }kG\textup{-modules } Y \textup{ and }T\in\mathcal{T}\\
    &\iff \tachom_{kG}(Y,T^*\otimes X) = 0\textup{ for all }kG\textup{-modules } Y \textup{ and }T\in\mathcal{T}\\
    &\iff T^*\otimes X \cong 0 \textup{ for all }T \in \mathcal{T}\\
    &\iff T\otimes X \cong 0\textup{ for all }T \in \mathcal{T}
\end{align}
The first equivalence follows from \Cref{lop}. The second equivalence is adjunction, using that $T$ is dualisable. The third step follows because we may choose $Y = T^*\otimes X$. The final step follows by \Cref{du} and the fact that for a dualisable $T$ we have $T \cong T^{**}$. This proves the claim, and hence the statement. 
\end{proof}
In particular, this means the map $\textup{Thickid}(G^d) \to \textup{Smashid}(G)$ defined by $\mathcal{T} \mapsto \locm(\mathcal{T})$ is well defined. 
\begin{definition}
    We say that the telescope conjecture holds for $G$ if the map $\textup{Thickid}(G^d) \to \textup{Smashid}(G)$ is a bijection.
\end{definition} 
We note that this is essentially what Wolcott \cite{WolcottTel} refers to as the strongly dualisable generalised smashing conjecture, although Wolcott assumes that the tensor unit generates the category as a localising subcategory, which is not true in our situation. However, it seems to be the correct generalisation of the telescope conjecture in this context, since the compact objects will not be closed under the tensor product. 
\par
As previously mentioned, in the rigidly-compactly generated case the above map is always injective. In \Cref{notele} and \Cref{notelex} we give examples to show that this map is not injective in general, and hence this provides examples where the telescope conjecture does not hold. 
\section{Infinite free products of finite groups}\label{freeprodsec}
In this section we consider the case where $G = \bigast\limits_{n \in \mathbb{N}} H_n$, and each $H_n$ is a finite group. As before, we let $k$ be a field of characteristic $p > 0$, and to avoid triviality we assume that $p$ divides the order of each $H_n$. 
\par 
We begin with identifying the compact and dualisable objects. 
\begin{lemma}\label{riginpro}
    A module $M$ is dualisable if and only if $M{\downarrow}_{H_n}^G$ is dualisable for each $n$. 
\end{lemma}
\begin{proof}
    This follows from \cite[Proposition~5.8]{kendall2024costratificationstablecategoriesinfinite}, which shows that a module $M$ is dualisable if and only if $M{\downarrow}_F^G$ is dualisable for every finite subgroup $F \leq G$. This shows the statement since every finite subgroup is conjugate to a subgroup of one of the $H_n$.
\end{proof}
\begin{lemma}\label{compinri}
    Every compact object is dualisable. 
\end{lemma}
\begin{proof}
    It suffices to show that each compact generator is dualisable. The category is generated by modules of the form $k{\uparrow}_{H_n}^G$ as shown in \Cref{comkin}.  By the Mackey formula, we have that $k{\uparrow}_{H_n}^G{\downarrow}_{H_n}^G \cong k$ and $k{\uparrow}_{H_n}^G{\downarrow}_{H_m}^G \cong 0$ if $m \neq n$. Therefore, from \Cref{riginpro} we see that $k{\uparrow}_{H_n}^G$ is dualisable.
\end{proof}
\begin{proposition}\label{produccomp}
    An object $X$ is compact if and only if $X{\downarrow}_{H_n}^G \neq 0$ for only finitely many $n \in \mathbb{N}$ and $X$ is dualisable.
\end{proposition}
\begin{proof}
    First suppose that $X$ is compact. In view of \Cref{compinri} we only need to show that $X{\downarrow}_{H_n}^G \neq 0$ for only finitely many $n \in \mathbb{N}$. It follows from \cite[Proposition~1.1.2]{KPLat} that there exist a finite set of compact generators $\mathcal{C}$ such that $X \in \textup{Thick}(\mathcal{C})$. This means that there exists a finite subset $\mathcal{N}\subseteq \mathbb{N}$ such that $X \in \textup{Thick}(\{k{\uparrow}_{H_n}^G\ |\ n \in \mathcal{N}\})$. 
    \par 
    Let $\mathcal{T} = \textup{Thick}_{\otimes}(\{k{\uparrow}_{H_n}^G\ |\ n \in \mathcal{N}\})$ and take $Y \in \mathcal{T}$. We claim that for any $m \not\in\mathcal{N}$ we have $Y{\downarrow}_{H_m}^G \cong 0$. The Mackey formula shows that this is true for all modules of the form $k{\uparrow}_{H_n}^G$ for $n \in\mathbb{N}$. It is easy to see that the claim is then true for any module in the thick tensor ideal generated by all such modules, and so it is true for any $Y \in \mathcal{T}$.
    \par 
    In particular, we have $X{\downarrow}_{H_m}^G \cong 0$ for every $m \not\in\mathcal{N}$, which shows that $X{\downarrow}_{H_n}^G \neq 0$ for only finitely many $n \in \mathbb{N}$. 
    \par 
    On the other hand, suppose that $X$ is dualisable and $X{\downarrow}_{H_n}^G \neq 0$ for only finitely many $n \in \mathbb{N}$. From the action of $G$ on a tree, we have a short exact sequence $0 \to P \to \bigoplus k{\uparrow}_{H_n}^G \to k \to 0$, with $P$ projective, because all the edge stabilisers are trivial. Therefore, we have a stable isomorphism $X \cong \bigoplus\limits_{n\in\mathbb{N}}X{\downarrow}_{H_n}^G{\uparrow}_{H_n}^G$. By assumption, we see that this is stably isomorphic to a finite direct sum. Furthermore, $X$ is dualisable and so from \Cref{riginpro} we know that $X{\downarrow}_{H_n}^G$ is dualisable, and therefore compact, from which it follows that $X{\downarrow}_{H_n}^G{\uparrow}_{H_n}^G$ is compact. Therefore, $X$ is stably isomorphic to a finite direct sum of compact objects and so is itself compact. 
\end{proof}
In particular, we have the following. 
\begin{corollary}
    There does not exist a single compact generator for $\tacstab(kG)$. Furthermore, the tensor unit is not compact.
\end{corollary}
\begin{proof}
    A compact generator would have to restrict to a nonzero object for infinitely many $n \in \mathbb{N}$, contradicting \Cref{produccomp}. Similarly, the tensor unit is simply $k$ which we can see is not compact. 
\end{proof}
In this special case, the tensor product is well defined on the subcategory of compact objects, as well as on the subcategory of dualisable objects.
\begin{proposition}\label{ttcompcatex}
    The subcategory of dualisable objects forms a tensor-triangulated category, while the subcategory of compact objects forms a non-unital tensor-triangulated category. 
\end{proposition}
\begin{proof}
    We first show that the tensor product of two dualisable objects is again dualisable. Take $M$ and $N$ both dualisable. We need to show that $(M \otimes N){\downarrow}_{H_n}^G$ is dualisable, but this is isomorphic to $M{\downarrow}_{H_n}^G \otimes N{\downarrow}_{H_n}^G$ which is clearly dualisable as desired. 
    \par 
   For the second claim it suffices to show that the tensor product of a compact object with a dualisable object is itself compact. This is not hard to see using the definitions, for example see \cite[Lemma~2.3]{HallRydh}. 
\end{proof}
We show that the telescope conjecture will not hold for many infinite free products. 
\begin{theorem}\label{notele}
    Let $G = \bigast\limits_{n \in \mathbb{N}} H_n$ where each $H_n$ is a finite $p$-group. Then the natural map $\textup{Thickid}(G^d) \to \textup{Smashid}(G)$ is not injective. 
\end{theorem}
\begin{proof}
    We claim that $\textup{Thick}_{\otimes}(\{k{\uparrow}_{H_n}^G\}) \subset \textup{Thick}_{\otimes}(k)$ is a strict inclusion. From this the result follows, since $\{k{\uparrow}_{H_n}^G\}_{n \in \mathbb{N}}$ form a set of compact generators, which follows just as \Cref{comkin}, and so generate the whole category as a localising tensor ideal.
    \par 
    We know, as in the proof of \Cref{ttcompcatex}, that the tensor product of a compact object with a dualisable object is compact. We therefore see that any object $X\in \textup{Thick}_{\otimes}(\{k{\uparrow}_{H_n}^G\})$ is compact. From \Cref{produccomp}, we then know that $X{\downarrow}_{H_n}^G \neq 0$ for only finitely many $n \in \mathbb{N}$. However, $k$ restricts to a non-zero object on every $H_n$, and so $k \not\in\textup{Thick}_{\otimes}(\{k{\uparrow}_{H_n}^G\})$ as claimed.
\end{proof}
\begin{corollary}\label{NOTELHERE}
    The telescope conjecture does not hold for the groups $G$ as in \Cref{notele}. 
\end{corollary}
This case is special, however, since there is a map $\textup{Thickid}(G^c) \to \textup{Smashid}(G)$, which is bijective and so a form of the telescope conjecture does hold. 
\par 
We now give an example of a group where the smashing ideals are determined by neither compact nor dualisable objects. We essentially add in the minimal amount of amalgamation in order to show that compact objects aren't closed under tensoring, and so it does not even make sense to talk about thick tensor ideals of compact objects. 
\begin{proposition}
    Let $G = \bigast\limits_{n \in \mathbb{N}}H_n$, where $H_1 = H_2 = C_4\ast_{C_2}C_4$ and $H_n = C_2$ for $n > 2$. Then $k{\uparrow}_{H}^G \otimes k{\uparrow}_{H}^G$ is not compact, where $H$ is the amalgamated copy of $C_2$ in $H_1$. In particular, the compact objects of $\tacstab(kG)$ are not closed under the tensor product.
\end{proposition} 
\begin{proof}
    We have an isomorphism $k{\uparrow}_{H}^G \otimes k{\uparrow}_{H}^G \cong k{\uparrow}_{H}^G{\downarrow}_H^G{\uparrow}_H^G$. We claim that $k{\uparrow}_H^G{\downarrow}_H^G$ contains $\bigoplus\limits_{g\in [H\backslash H_1 /H]}k$ as a summand. By the Mackey formula, we know that $k{\uparrow}_H^{H_1}{\downarrow}_H^{H_1} \cong \bigoplus\limits_{g\in [H\backslash H_1 /H]}k$, using that $H_1$ is central in $H$. Again by the Mackey formula, we know that $k{\uparrow}_H^{H_1}$ is a summand of $(k{\uparrow}_H^{H_1}){\uparrow}_{H_1}^G{\downarrow}_{H_1}^G$ from which we infer the claim. 
    \par 
    It follows from this that $k{\uparrow}_{H}^G{\downarrow}_H^G{\uparrow}_H^G$ contains the infinite coproduct $\bigoplus\limits_{g\in [H\backslash H_1 /H]}k{\uparrow}_{H}^G$ as a summand, which is clearly not compact. 
\end{proof}
\begin{theorem}\label{notelex}
Let $G = \bigast\limits_{n \in \mathbb{N}}H_n$, where $H_1 = H_2 = C_4\ast_{C_2}C_4$ and $H_n = C_2$ for $n > 2$. Then the map $\textup{Thickid}(G^d) \to \textup{Smashid}(G)$ is not injective and so the telescope conjecture does not hold. 
\end{theorem}
\begin{proof}
Let $K = \bigoplus\limits_{n > 2}k{\uparrow}_{H_n}^G$. We see immediately that $\locm(K) = \locm(\{k{\uparrow}_{H_n}^G\}_{n > 2})$ and it follows from \Cref{riginpro} that $K$ is dualisable. We claim that $\textup{Thick}_{\otimes}(\{k{\uparrow}_{H_n}^G\}_{n > 2}) \subset \textup{Thick}_{\otimes}(K)$ is a strict inclusion. This would then show two distinct thick tensor ideals of dualisable objects which generate the same localising tensor ideal, showing the statement. 
\par 
    Indeed, we may argue as in \Cref{compinri} to see that anything in $\textup{Thick}_{\otimes}(\{k{\uparrow}_{H_n}^G\}_{n > 2})$ must restrict to a non-zero module on only finitely many $n > 2$. This is clearly not the case for $K$ and hence $K \not\in\textup{Thick}_{\otimes}(\{k{\uparrow}_{H_n}^G\}_{n > 2})$.
\end{proof}
Obviously, many more examples may be constructed in a similar way. 
\par
We now consider the Balmer spectrum. As suggested by the lack of telescope conjecture, this turns out to be a delicate question. 
\begin{proposition}\label{oneusrj}
    The natural map $\textup{Thickid}^f(G^d) \to \prod\limits_{n \in \mathbb{N}}\textup{Thickid}^f(H_n^d)$ is an epimorphism.
\end{proposition}
\begin{proof}
    Let $(\textup{Thick}_{\otimes}(M_n))_{n \in \mathbb{N}}$ be an element of the product. By the construction in \cite[Section~7]{MSstabcat}, we may find a $kG$-module $M$ such that $M{\downarrow}_{H_n}^G \cong M_n$ for each $n \in \mathbb{N}$. We see from \Cref{riginpro} that $M$ is dualisable, and altogether this shows the surjectivity of the map. We conclude that it is an epimorphism by \Cref{reflectingthings}.
\end{proof}
We turn now to injectivity of the map; in general we will show that this map is not always injective. However, we begin with the groups where we can show that it is indeed injective.
\par
Before we do this, we recall some notation and an important result from \cite{BalmSurj}. Let $M$ be a dualisable object and complete the coevaluation morphism $k \to M \otimes M^*$, i.e. the unit of the adjunction, to a triangle of the following form. 
\begin{equation}\label{ximeq}\begin{tikzcd}
     C_M \arrow{r}{\xi_M} & k \arrow{r} & M \otimes M^* \arrow{r} & \Omega^{-1}(C_M)\end{tikzcd}
\end{equation}
Since $\tacstab(kG)^d$ is a tensor-triangulated category where every object is dualisable, we may apply \cite[Corollary~2.11]{BalmSurj} to see that \begin{equation}\label{52}\textup{Thick}_{\otimes}(M) = \{X \in \tacstab(kG)^d\ |\ \xi_M^{\otimes n}\otimes X = 0 \textup{ for some } n \geq 0\}\end{equation}
\begin{lemma}\label{ximiszero}
    Let $H = C_p$ and $M$ be any non-projective finitely generated $kH$-module. Then $\xi_M = 0$. 
\end{lemma}
\begin{proof}
    We know that $\xi_M \otimes M = 0$ as shown in \cite[Lemma~7.5]{BENSONCARLSONGEN} for example. It is shown in \cite[Proposition~5.1]{ttrum} that any non-zero compact object $X$ in $\tacstab(kH)$ is $\otimes$-faithful, that is if $X \otimes f = 0$ for some morphism $f$, then $f = 0$. From this, we immediately conclude that $\xi_M = 0$. 
\end{proof}
Note that any non-zero compact object being $\otimes$-faithful is part of the definition of a tensor-triangular field \cite{ttrum}, and so a similar statement holds for any tensor-triangular field. 
\begin{proposition}\label{constantc}
    Suppose $H$ is a finite group of $p$-rank one. Then there exists a constant $C$ such that for any non-projective finitely generated $kH$-module $M$ we have $\xi_M^{\otimes C} = 0$.
\end{proposition}
\begin{proof}
    First, we note that the triangle \eqref{ximeq} is preserved by restriction, by which we mean that for a subgroup $F \leq H$ we have $\xi_M{\downarrow}_F^H = \xi_{M{\downarrow}_F^H}$. In particular, for any elementary abelian subgroup $F\leq H$, which is necessarily cyclic of prime order, we know from \Cref{ximiszero} that $\xi_M{\downarrow}_F^H = 0$. The result now follows from \cite[Theorem~2.4]{BENSONCARLSONGEN}.
\end{proof}
We note that this answers a question of Benson and Carlson \cite[Question~1.3]{BENSONCARLSONGEN} in this case. See \Cref{deftensorgenerat} for the definition of tensor generation.
\begin{corollary}
    Let $H$ be a finite group of $p$-rank one. Then there exists a constant $C$ depending only on $H$ such that for every non-projective finitely generated module $M$, any other finitely generated $kH$-module $N$ may be tensor generated from $M$ in at most $C$ steps. 
\end{corollary}
\begin{proof}
    This is immediate from \Cref{constantc} and \cite[Theorem~8.10]{BENSONCARLSONGEN}.
\end{proof}
We can now determine the spectrum for various groups.
\begin{theorem}\label{specfortrees}
    Let $\mathcal{H}$ be a finite collection of finite groups of $p$-rank one, and for each $n \in \mathbb{N}$ let $H_n \in\mathcal{H}$. Suppose $G = \bigast\limits_{n\in\mathbb{N}}H_n$. Then $\textup{Spc}(\tacstab(kG)^d)$ is homeomorphic to the Stone-\v{C}ech compactification of $\mathbb{N}$. 
\end{theorem}
\begin{proof}
   We begin by showing that the map $\textup{Thickid}^f(G^d) \to \prod\limits_{n \in \mathbb{N}}\textup{Thickid}^f(H_n^d)$ is injective, which implies by \Cref{reflectingthings} that it is a monomorphism. 
    \par 
    Suppose that $M$ and $X$ are $kG$-modules such that $\textup{Thick}_{\otimes}(M{\downarrow}_{H_n}^G) = \textup{Thick}_{\otimes}(X{\downarrow}_{H_n}^G)$ for each $n \in \mathbb{N}$. We claim that $\xi_M^{\otimes n}\otimes X = 0$ for some $n \geq 0$. Assuming this is true, we conclude from \cite[Corollary~2.11]{BalmSurj}, i.e. the equality in \eqref{52}, that $X \in \textup{Thick}_{\otimes}(M)$; reversing the roles of $X$ and $M$ shows that $\textup{Thick}_{\otimes}(M) = \textup{Thick}_{\otimes}(X)$ as desired. 
    \par 
    We now show the claim. Since $\mathcal{H}$ is a finite collection of groups, we may use \Cref{constantc} to find some constant $C$ such that for any $H \in \mathcal{H}$ and any finitely generated $kH$-module $N$, which is not projective, we have $\xi_N^{\otimes C} = 0$. If $N$ is projective, and $Y \in \textup{Thick}_{\otimes}(N)$, then $Y$ is projective and so $\xi_N \otimes Y = 0$. 
    \par 
    Therefore, we know that $(\xi_M^{\otimes C}\otimes X){\downarrow}_{H_n}^G = 0$ for each $H_n$. Note here that we are using that $M{\downarrow}_{H_n}^G$ is (stably) finitely generated by \Cref{riginpro}. It now follows from \cite[Theorem~3.3]{gomez2024picardgroupstablemodule} that $\xi_M^{\otimes C}\otimes X = 0$ as desired. 
    \par 
    We have shown that the map $\textup{Thickid}^f(G^d) \to \prod\limits_{n \in \mathbb{N}}\textup{Thickid}^f(H_n^d)$ is injective and we know from \Cref{oneusrj} that it is surjective; we conclude that this map is bijective in the category of sets. We know from \Cref{reflectingthings} that this map is therefore an isomorphism of distributive lattices. Furthermore, we know from \cite[Lemma~3.14]{barthel2023descenttensortriangulargeometry} that the forgetful functor from the category of distributive lattices to the category of sets creates limits. 
     \par 
     From this, we conclude that $\textup{Thickid}^f(G^d)$ is isomorphic to the product $\prod\limits_{n \in \mathbb{N}}\textup{Thickid}^f(H_n^d)$ in the category of distributive lattices. Stone duality as in \Cref{Stonedual} then shows that $\textup{Spc}(\tacstab(kG)^d)$ is homeomorphic to the spectral coproduct $\bigsqcup\limits_{n\in\mathbb{N}}\textup{Spc}(\tacstab(kH_n)^d)$. 
    \par 
    Now, we know that $\textup{Spc}(\tacstab(kH_n)^d) = \{\ast\}$. Therefore, the spectrum of dualisable objects in $\tacstab(kG)$ is the spectral coproduct $\bigsqcup\limits_{n\in\mathbb{N}}\{\ast\}$; this is shown in \cite[Example~10.1.5]{Dickmann_Schwartz_Tressl_2019} to be the Stone-\v{C}ech compactification of $\mathbb{N}$ as claimed.
\end{proof}
The above situation is very special; one of the key properties is that the spectrum of the stable category of a cyclic group is a point. As soon as we have a non-trivial, proper thick tensor ideal in each $H_n$, we cannot determine the spectrum of the stable category for the infinite free product in this way, even as a spectral space. 
\par 
We need some preparation in order to show this; for the following definitions see \cite[Section~8]{BENSONCARLSONGEN} and \cite[Section~2]{SteSte}. 
\begin{definition}\label{deftensorgenerat}
    Let $M$ and $X$ be $kG$-modules. We say that $X$ is generated from $M$ in one step if $X$ is a summand of a direct sum of shifts of $M$. We then inductively say that $X$ is generated from $M$ in $n$ steps if there exists a triangle 
    \[N_1 \to N_2 \to X \oplus Z \to\]
    with $N_1$ generated from $M$ in $n-1$ steps, $N_2$ generated from $M$ in one step, and $Z$ arbitrary. 
    \par 
    We say $X$ is tensor generated from $M$ in $n$ steps if $X$ is generated from modules of the form $M\otimes Y$ in $n$ steps, where $Y$ is arbitrary. 
    \par
    Finally, we say a thick tensor ideal $\mathcal{T}$ is strongly generated by $M$ if every module $X \in \mathcal{T}$ is generated from $M$ in at most $n$ steps, for some fixed $n \geq 1$.
\end{definition}
\begin{lemma}\label{tengenmsteps}
    Let $H$ be a finite group and suppose $M$ and $X$ are finitely generated $kH$-modules such that $\xi_M^{\otimes n}\otimes X = 0$ for some $n$. Then $M$ tensor generates $X$ in at most $n$ steps.
\end{lemma}
\begin{proof}
    This is a generalisation of \cite[Proposition~8.9]{BENSONCARLSONGEN}. If $\xi_M^{\otimes n}\otimes X = 0$ for some $n$, then we know that $X$ is a summand of $\textup{cone}(\xi_M^{\otimes n})\otimes X$. Therefore, it suffices to show that $M$ tensor generates $\textup{cone}(\xi_M^{\otimes n})$ in at most $n$ steps. 
    \par 
    We show this by induction. The case $n = 1$ is clear, since $\textup{cone}(\xi_M) = M \otimes M^*$ by definition. So, assume $n > 1$ and that it holds for all $i < n$. For simplicity, we write $\xi_M: N \to k$. We have that $\xi_M^{\otimes n} = (\xi_M \otimes \textup{id}_{k^{\otimes n-1}})\circ(\textup{id}_{N}\otimes \xi_M^{\otimes n-1})$. Applying the octahedral axiom to the obvious triangles resulting from this relation gives an exact triangle of the form $N \otimes \textup{cone}(\xi_M^{\otimes n-1}) \to \textup{cone}(\xi_M^{\otimes n}) \to \textup{cone}(\xi_M) \to$, and the inductive hypothesis proves the claim. 
\end{proof}
\begin{lemma}\label{loeqy}
    Let $H$ be a finite $p$-group and suppose we have finitely generated $kH$-modules $M$ and $X$ such that $M$ tensor generates $X$ in $n$ steps. Then $M$ generates $X$ in at most $\ell n$ steps, where $\ell$ is the radical length of $kH$.
\end{lemma}
\begin{proof}
    This follows just as \cite[Lemma~9.1]{BENSONCARLSONGEN}. The point is that for any finitely generated $Y$, the radical filtration of $Y$ has length at most $\ell$ and each quotient is isomorphic to a direct sum of copies of $k$. From this, we see that $M \otimes Y$ may be generated from $M$ in at most $\ell$ steps. 
\end{proof}
Given this, the following is essentially a consequence of \cite[Theorem~4.1]{SteSte}, which shows that, under quite general conditions, the only thick tensor ideals to have strong generators are the zero ideal and the full category.
\begin{theorem}\label{NOGENERATE}
    Suppose $H$ is a finite $p$-group such that there exists a non-trivial, proper thick tensor ideal of the stable category. Then for each $n \in \mathbb{N}$ there exists finitely generated $kH$-modules $M$ and $X$ such that $M$ tensor generates $X$ in greater than $n$ steps. 
\end{theorem}
\begin{proof}
    Let $\mathcal{T}$ be a non-trivial proper thick tensor ideal of the stable category of finitely generated $kH$-modules and take $M \in \mathcal{T}$. Note that $\textup{Thick}_{\otimes}(M)$ is itself a non-trivial proper thick tensor ideal. Hence, from \cite[Theorem~4.1]{SteSte}, we know that $\textup{Thick}_{\otimes}(M)$ does not have a strong generator.
    \par 
    Take $n \in \mathbb{N}$ and suppose that for all $X \in \textup{Thick}_{\otimes}(M)$, we have that $X$ can be tensor generated from $M$ in $\leq n$ steps. From \Cref{loeqy} we see that $X$ can be generated from $M$ in $\leq \ell n$ steps; this implies that $M$ is a strong generator of $\textup{Thick}_{\otimes}(M)$, which is a contradiction. This shows that there must exist some $X$ which is tensor generated from $M$ in greater than $n$ steps.
\end{proof}
\begin{theorem}\label{NOSPECFORTREES}
    Let $G = \bigast\limits_{n\in\mathbb{N}}H_n$ where $H_n$ is a finite $p$-group of $p$-rank at least two for each $n \in \mathbb{N}$. Then the natural map $\textup{Thickid}^f(G^d) \to \prod\limits_{n \in \mathbb{N}}\textup{Thickid}^f(H_n^d)$ is not a monomorphism of distributive lattices.
\end{theorem}
\begin{proof}
    First note that for a finite group of $p$-rank at least two, we know from Benson, Carlson, and Rickard \cite{bcrthick} that there exists a non-trivial proper thick tensor ideal of the stable category. For each $n$, we let $M_n$ and $X_n$ be $kH_n$ modules as in \Cref{NOGENERATE}, so $X_n \in \textup{Thick}_{\otimes}(M_n)$ but $M_n$ tensor generates $X_n$ in greater than $n$ steps. By \cite[Section~7]{MSstabcat}, we may find $kG$-modules $M$ and $X$ such that $M{\downarrow}_{H_n}^G \cong M_n$ and $X{\downarrow}_{H_n}^G \cong X_n$ for each $n$. We see from \Cref{riginpro} that $M$ and $X$ are both dualisable. 
    \par 
    We know from construction that $X{\downarrow}_{H_n}^G \in \textup{Thick}_{\otimes}(M{\downarrow}_{H_n}^G)$ for each $n$, and we claim that $X \not\in\textup{Thick}_{\otimes}(M)$. By Balmer's result \eqref{52}, it suffices to show that $\xi_M^{\otimes n}\otimes X \neq 0$ for any $n$. Suppose not, and so for some $d > 0$ we have $\xi_M^{\otimes d}\otimes X = 0$. Choose any $l > d$ and we see that $0=(\xi_M^{\otimes d}\otimes X){\downarrow}_{H_l}^G = \xi^{\otimes d}_{M_l}\otimes X_l$. This implies from \Cref{tengenmsteps} that $M_l$ tensor generates $X_l$ in at most $d$ steps; however since $l > d$ this contradicts our choice of $M_l$ and $X_l$. 
    \par 
    Consider the triangle 
    \begin{equation}\label{ximeq2}\begin{tikzcd}
     C_M\otimes X \arrow{r}{f} & X \arrow{r} & \textup{cone}(f) \arrow{r} & \Omega^{-1}(C_M\otimes X)\end{tikzcd}
\end{equation}
    where we have written $f = \xi_M\otimes X$ for notation ease. Note that $\textup{cone}(f) = M \otimes M^* \otimes X$.
    \par 
    Now, take some $H_n$ and consider the restriction of the triangle to $H_n$. Since $X{\downarrow}_{H_n}^G \in \textup{Thick}_{\otimes}(M{\downarrow}_{H_n}^G)$, we know from \eqref{52} that $(f{\downarrow}_{H_n}^G)^{\otimes m} = (\xi_{M{\downarrow}_{H_n}^G} \otimes X{\downarrow}_{H_n}^G)^{\otimes m} = 0$ for some $m \geq 1$. This implies that the triangle splits, and so $\textup{Thick}_{\otimes}(\textup{cone}((f{\downarrow}_{H_n}^G)^{\otimes m})) = \textup{Thick}_{\otimes}((X{\downarrow}_{H_n}^G)^{\otimes m} \oplus ((C_M \otimes M){\downarrow}_{H_n}^G)^{\otimes m})$. 
    \par 
    Note that $(f{\downarrow}_{H_n}^G) \otimes \textup{cone}((f{\downarrow}_{H_n}^G)) = 0$ because $\textup{cone}((f{\downarrow}_{H_n}^G)) = (M^*\otimes M\otimes X){\downarrow}_{H_n}^G$ and $\xi_{M{\downarrow}_{H_n}^G} \otimes M{\downarrow}_{H_n}^G = 0$ by the unit-counit relation, see for example the proof of \cite[Corollary~2.11]{BalmSurj}. It then follows from \cite[Proposition~2.10]{BalmSurj} that $\textup{Thick}_{\otimes}(\textup{cone}(f{\downarrow}_{H_n}^G)) = \textup{Thick}_{\otimes}(\textup{cone}((f{\downarrow}_{H_n}^G)^{\otimes m}))$. Furthermore, since $C_M$ and $X$ are dualisable, we know that $\textup{Thick}_{\otimes}((X{\downarrow}_{H_n}^G)^{\otimes m}\oplus ((C_M \otimes X){\downarrow}_{H_n}^G)^{\otimes m}) = \textup{Thick}_{\otimes}((X\oplus C_M \otimes X){\downarrow}_{H_n}^G)$.
    \par 
    In particular, we have $\textup{Thick}_{\otimes}(\textup{cone}(f){\downarrow}_{H_n}^G) = \textup{Thick}_{\otimes}((X \oplus C_M \otimes X){\downarrow}_{H_n}^G)$ for each $n \in \mathbb{N}$. Note that we have used the fact that $\textup{cone}(f{\downarrow}_{H_n}^G) = \textup{cone}(f){\downarrow}_{H_n}^G$.
    \par 
    We claim that $\textup{Thick}_{\otimes}(\textup{cone}(f)) \neq \textup{Thick}_{\otimes}(X \oplus C_M \otimes X)$. More specifically, we show that $X \notin\textup{Thick}_{\otimes}(\textup{cone}(f))$. 
    \par 
   The fact that $f \otimes \textup{cone}(f) = 0$ implies that $\textup{Thick}_{\otimes}(\textup{cone}(f)) \subseteq \{Z\ |\ f^{\otimes m}\otimes Z = 0 \textup{ for some } m \geq 1\}$, noting that the latter subcategory is a thick tensor ideal \cite[Proposition~2.9]{BalmSurj}. So, if $X \in \textup{Thick}_{\otimes}(\textup{cone}(f))$, then $f^{\otimes m}\otimes X = 0$ for some $m \geq 1$. However, recall that $f = \xi_M \otimes X$ and so this implies by \eqref{52} that $X^{\otimes m + 1} \in \textup{Thick}_{\otimes}(M)$. Since $X$ is dualisable, this implies that $X \in \textup{Thick}_{\otimes}(M)$, which is a contradiction.
    \par 
    Altogether, this shows that the map is not injective in the category of sets. We conclude from \Cref{reflectingthings} that the map is not a monomorphism of distributive lattices. 
\end{proof}
\begin{remark}
    Under Stone duality, the map of \Cref{NOSPECFORTREES} corresponds to a map from the spectral coproduct $\bigsqcup\limits_{n\in\mathbb{N}}\textup{Spc}(\tacstab(kH_n)^d) \to \textup{Spc}(\tacstab(kG)^d)$, which is induced by restriction. \Cref{NOSPECFORTREES} shows that this map is not an epimorphism. Note that the restriction functors are jointly conservative, that is $M \cong 0$ if and only if $M{\downarrow}_{H_n}^G \cong 0$ for each $H_n$. We see that the rigidly-compactly generated hypothesis in \cite[Theorem~1.3]{surjinttgeom} is strictly necessary, even accounting for the possibility of taking spectral coproducts rather than topological coproducts.
\end{remark}
We can use \Cref{specfortrees} to give an example of a tensor-triangulated category which is not stratified by the Balmer spectrum of dualisable objects, in the sense of \cite[Definition~10.30]{barthel2023cosupport}. For the rest of the section, we let $G=  \bigast\limits_{n \in \mathbb{N}} H_n$ where $H_n = C_p$ for each $n$, for simplicity.
\par 
We briefly recall what it means to be stratified. Let $\mathscr{T}$ be a perfectly generated, non-closed tensor-triangulated category and assume that $\mathscr{T}$ is weakly noetherian. It is shown in \cite[Definition~10.24]{barthel2023cosupport} that there exists a Balmer-Favi type support, which takes values in the spectrum of dualisable objects. We then say that $\mathscr{T}$ is stratified if this support induces a bijection between the localising $\mathscr{T}^d$-submodules of $\mathscr{T}$ and subsets of $\textup{Spc}(\mathscr{T}^d)$. 
\par 
Note that a localising $\mathscr{T}^d$-submodule is a localising subcategory closed under tensoring with all dualisable objects. We first show that the category of interest fits into this setup. 
\begin{lemma}
    The stable category is a perfectly generated, non-closed tensor-triangulated category with $\textup{Spc}(\tacstab(kG)^d)$ weakly noetherian. Furthermore, every localising $\tacstab(kG)^d$-submodule is in fact a localising tensor ideal. 
\end{lemma}
\begin{proof}
    We have described the spectrum of dualisable objects in \Cref{specfortrees}. To see that this is weakly noetherian, first we recall from \cite[Remark~2.4]{BHSstrat} that every $T_1$ spectral space is weakly noetherian. It remains to note that the Stone-\v{C}ech compactification is, by definition, Hausdorff and hence $T_1$.
    \par 
    For the statement about localising $\tacstab(kG)^d$-submodules, it suffices to recall from \Cref{compinri} that each compact object is dualisable and so every localising $\tacstab(kG)^d$-submodule is closed under tensoring with compacts, and hence closed under tensoring with any module. 
\end{proof}
\begin{lemma}\label{lococr}
    The set of localising tensor ideals of $\tacstab(kG)$ is in one to one correspondence with subsets of $\mathbb{N}$.
\end{lemma}
\begin{proof}
First, we recall that for a cyclic group $H_n$, the spectrum of $\tacstab(kH_n)^d$ is a single point. We know from \cite[Theorem~5.10]{kendall2024costratificationstablecategoriesinfinite} that the set of localising tensor ideals is in one to one correspondence with subsets of $\bigsqcup\limits_{n\in\mathbb{N}}\{\ast\}$, which we may view as $\mathbb{N}$ with the discrete topology.  
\end{proof}
\begin{theorem}\label{NOSTRAT}
    The category $\tacstab(kG)$ is not stratified by its spectrum of dualisable objects. 
\end{theorem}
\begin{proof}
To be stratified means that the Balmer-Favi support gives a bijection from the localising tensor ideals of $\tacstab(kG)$ to subsets of $\textup{Spc}(\tacstab(kG)^d)$. However, we see from \Cref{lococr} that the set of localising tensor ideals has cardinality $2^{\aleph_0}$, whereas from \Cref{specfortrees} we know that the spectrum is the Stone-\v{C}ech compactification of $\mathbb{N}$, which has cardinality $2^{2^{\aleph_0}}$. Therefore, no such bijection can exist. 
\end{proof}
We can generalise this as follows. 
\begin{theorem}\label{NOSPEC22}
    Let $\mathcal{T} = \prod\limits_{n\in\mathbb{N}}\mathcal{T}_n$, where each $\mathcal{T}_n$ is a tensor-triangular field. Then $\textup{Spc}(\mathcal{T}^d)$ is homeomorphic to the Stone-\v{C}ech compactification of $\mathbb{N}$, and $\mathcal{T}$ is not stratified by its spectrum of dualisable objects. 
\end{theorem}
\begin{proof}
    Note that the map $\textup{Thickid}^f(\mathcal{T}^d) \to \prod\limits_{n\in\mathbb{N}}\textup{Thickid}^f(\mathcal{T}_n^d)$ is clearly surjective. To see that it is injective, we can argue just as \Cref{specfortrees}. The key is that if $X\in \mathcal{T}_n$ is a non-zero dualisable object, then $\xi_X = 0$, using that $X$ is $\otimes$-faithful by definition of a tensor-triangular field and that $\xi_X \otimes X = 0$, which is shown in the proof of \cite[Corollary~2.11]{BalmSurj}. If $X \cong 0$ in $\mathcal{T}_n$, then for any $Y \in \textup{Thick}_{\otimes}(X)$, we have that $Y \cong 0$ and so $\xi_X\otimes Y = 0$ in this case. 
    \par 
    Suppose that $X,Y \in \mathcal{T}$ are such that $\textup{Thick}_{\otimes}(X_n) = \textup{Thick}_{\otimes}(Y_n)$, where we write $X_n$ is the image of $X$ in $\mathcal{T}_n$ under the canonical projection map, and similarly for $Y_n$. By the previous paragraph, we conclude that $\xi_{X_n} \otimes Y_n = 0$ for each $n \in \mathbb{N}$, and so $\xi_X \otimes Y = 0$. This implies by \cite[Corollary~2.11]{BalmSurj} that $Y \in \textup{Thick}_{\otimes}(X)$. Reversing the roles of $X$ and $Y$ shows that $\textup{Thick}_{\otimes}(X) = \textup{Thick}_{\otimes}(Y)$. 
    \par We conclude from \Cref{reflectingthings} that the map is an isomorphism, and so there is an isomorphism of spectral spaces $\bigsqcup\limits_{n\in\mathbb{N}}\textup{Spc}(\mathcal{T}_n^d) \to \textup{Spc}(\mathcal{T}^d)$. The spectrum of a tensor-triangular field is a point \cite[Proposition~5.15]{ttrum} and so $\textup{Spc}(\mathcal{T}^d)$ is the spectral coproduct of points. This is the Stone-\v{C}ech compactification of $\mathbb{N}$ as shown in \cite[Example~10.1.5]{Dickmann_Schwartz_Tressl_2019}.
    \par 
    To see that the category is not stratified, first note that since the spectrum of each $\mathcal{T}_n$ is a point, the localising tensor ideals are in one to one correspondence with $\{0,1\}$. We conclude from \cite[Lemma~3.17]{kendall2024costratificationstablecategoriesinfinite} that the set of localising tensor ideals is in one to one correspondence with $\prod\limits_{n\in\mathbb{N}}\{0,1\}$. This is in bijection with the powerset of $\mathbb{N}$ and so has cardinality $2^{\aleph_0}$. An argument on cardinalities as in \Cref{NOSTRAT} shows that the category cannot be stratified.
\end{proof}
\begin{remark}\label{remarkprod}
    In fact, the arguments of \Cref{NOSPECFORTREES} also work more generally. Inspecting the arguments, we see that we need $\mathcal{T} = \prod\limits_{n\in\mathbb{N}}\mathcal{T}_n$, where each $\mathcal{T}_n$ is a rigidly-compactly generated tensor-triangulated category with connected Balmer spectrum whose compact part is strongly generated by the tensor unit and which contains a non-trivial, proper thick tensor ideal; note that these conditions are required in order to be able to apply the results of \cite{SteSte}. Under these conditions, the projection functors $p_n: \mathcal{T} \to \mathcal{T}_n$ form a jointly conservative family of functors which induce a monomorphic but not epimorphic map $\bigsqcup\limits_{n\in\mathbb{N}}\textup{Spc}(\mathcal{T}_n^d) \to \textup{Spc}(\mathcal{T}^d)$.
\end{remark}
\section{$\mathrm{H}_1\mathfrak{F}$ groups of type $\mathrm{FP}_{\infty}$}
Throughout this section we let $G$ be a $\mathrm{H}_1\mathfrak{F}$ group of type $\mathrm{FP}_{\infty}$ and $k$ be a field of characteristic $p > 0$. Recall that a group $G$ is of type $\textup{FP}_{\infty}$ if the trivial module $k$ has a projective resolution consisting of finitely generated projective modules. 
\begin{lemma}
    The tensor unit is compact.
\end{lemma}
\begin{proof}
    This follows from \cite[Proposition~5.19]{kendall2024stablemodulecategorymodel}.
\end{proof}
Note that this implies that every dualisable object is compact. 
\begin{proposition}\label{fiszero}
    Suppose $f:M \to N$ is a morphism in $\tacstab(kG)$ with $M$ and $N$ dualisable, which is such that $f{\downarrow}_E^G = 0$ for each finite elementary abelian subgroup $E \leq G$. Then $f^{\otimes n} = 0$ for some $n \geq 0$.
\end{proposition}
\begin{proof}
    This follows from \cite[Theorem~13.3]{bensoninf}. We provide a different proof. First, note that we know from \cite[Proposition~2.9]{BalmSurj} that the full subcategory of $\tacstab(kG)$ consisting of modules $X$ such that $f^{\otimes n}\otimes X = 0$ for some $n \geq 0$ is a thick tensor ideal. By assumption it contains $k{\uparrow}_E^G$ for each elementary abelian subgroup $E \leq G$ and so, by \Cref{comkin}, it must also contain $k$.  
\end{proof}
The key input in determining the Balmer spectrum is a result of Henn \cite{Hennq}, which shows that Quillen Stratification holds for such groups.
\begin{definition}
    A homomorphism $\phi: A \to B$ of $k$-algebras is called a uniform $F$-isomorphism if there exists a natural number $n$ such that if $x$ is in $\textup{ker}(\phi)$ then $x^{p^n} = 0$ and for every $y \in B$ we have $y^{p^n} = \phi(y')$ for some $y' \in A$.
\end{definition}
From now on we let $\mathcal{E}(G)$ be the collection of elementary abelian $p$-subgroups of $G$. For the following, recall the definition of the category $\mathcal{A}_{\mathcal{E}}(G)$ from \Cref{aegcatdef}.
\begin{theorem}\label{thm:henn} \cite[Appendix]{Hennq} 
    Let $G$ be an $\textup{H}_1\mathfrak{F}$ group of type $\textup{FP}_{\infty}$. Then the natural map $H^*(G,k) \to \lim\limits_{E \in \mathcal{A}_{\mathcal{E}}(G)} H^*(E,k)$ is a uniform $F$-isomorphism, there are only finitely many conjugacy classes of finite elementary abelian $p$-subgroups, and $H^*(G,k)$ is a finitely generated algebra over $k$.
\end{theorem}
As a consequence, we have the following cf. \cite[Section~12]{quill2}.
\begin{corollary}\label{cor:fisohomeo}
    For $G$ an $\textup{H}_1\mathfrak{F}$ group of type $\textup{FP}_{\infty}$ there is a homeomorphism \[\underset{{E \in \mathcal{A}_{\mathcal{E}}(G)}}{\textup{colim}}\textup{Proj}(H^*(E,k)) \to \textup{Proj}(H^*(G,k))\]
\end{corollary}
\begin{remark}\label{supremark}
    For $G$ an $\textup{H}_1\mathfrak{F}$ group of type $\textup{FP}_{\infty}$, we may define our support to take values in $\textup{Proj}(H^*(G,k))$ by combining the support we defined in \Cref{subsectionofsupport} with the homeomorphism of \Cref{cor:fisohomeo}. We see that under this identification, the support of $M$ is simply the union over all elementary abelian $p$-subgroups $E \leq G$ of $\textup{res}^*_{G,E}(\textup{Supp}_E(M{\downarrow}_E^G))$, where $\textup{res}^*_{G,E}$ is the map induced by the subgroup inclusion $E \leq G$. 
\end{remark}
Mimicking the construction of Carlson's $L_{\zeta}$ modules for finite groups \cite[Section~5.9]{bensonbook} allows us to realise closed subsets of $\textup{Proj}(H^*(G,k))$ as the support of dualisable objects. 
\begin{proposition}\label{carlsonzeta}
    Suppose $\mathcal{V}$ is a closed subset of $\textup{Proj}(H^*(G,k))$. Then $\mathcal{S} = \textup{Supp}_G(M)$ for a dualisable object $M$. 
\end{proposition}
\begin{proof}
First, suppose that $\zeta \in H^*(G,k)$ is a homogeneous element of degree $n$ and hence corresponds to a map $\zeta: \Omega^n(k) \to k$. This in turn corresponds to a morphism in $\tacstab(kG)$, and completing this to a triangle gives an exact triangle of the form 
\[\begin{tikzcd}
L_{\zeta} \arrow{r}& \Omega^n(k) \arrow{r}{\zeta} & k \arrow{r} & \Omega^{-1}(L)\end{tikzcd}\]
For any finite subgroup $F \leq G$, restricting this triangle to $F$ is isomorphic to the exact triangle corresponding to $\zeta{\downarrow}_F^G$, i.e we have a stable isomorphism $L_{\zeta}{\downarrow}_F^G \cong L_{\zeta{\downarrow}_F^G}$. It follows from \cite[Lemma~2.6]{BIKstrat} that $\textup{Supp}_F(L_{\zeta{\downarrow}_F^G}) = V(\mathcal{\zeta}{\downarrow}_F^G)$, the set of homogeneous prime ideals containing $\zeta{\downarrow}_F^G$. Therefore, we see that $\textup{Supp}_F(L_{\zeta}{\downarrow}_F^G) = V(\zeta{\downarrow}_F^G)$.
\par 
We claim that this implies that $\textup{Supp}_G(L_{\zeta}) = V(\zeta)$. The key input is \Cref{cor:fisohomeo} since from this we see that every $\mathfrak{p} \in \textup{Proj}(H^*(G,k))$ is equal to $\textup{res}^*_{G,E}(\mathfrak{q})$ for some $E \leq G$ an elementary abelian subgroup and $\mathfrak{q} \in \textup{Proj}(H^*(E,k))$. It follows that $\zeta \in \mathfrak{p}$ if and only if $\zeta{\downarrow}_E^G \in \mathfrak{q}$. Hence, we have the equality $V(\zeta) = \bigcup\limits_{E \in \mathcal{E}(G)}\textup{res}^*_{G,E}(V(\zeta{\downarrow}_E^G))$. By what we have shown above, and \Cref{supremark}, the right hand side is exactly $\textup{Supp}_G(L_{\zeta})$.
\par 
    Now, suppose $\mathcal{V} = V(I)$, the Zariski closure of some homogeneous ideal $I$. Then $I$ is generated by homogeneous elements $\zeta_1,\dots,\zeta_n$, which is a finite set of generators by \Cref{thm:henn}. From the tensor product property of support given in \Cref{prop:supportprop}, and what we have shown above, it follows that $\textup{Supp}_G(L_{\zeta_1} \otimes \dots \otimes L_{\zeta_n}) = \textup{Supp}_G(L_{\zeta_1}) \cap \dots \cap\textup{Supp}_G(L_{\zeta_n}) = \mathcal{V}$ as desired.
\end{proof}
Let $E\leq F \leq G$ be elementary abelian subgroups and consider some $\textup{Thick}_{\otimes}(M) \in \textup{Thickid}^f(F^d)$. We know from \cite[Remark~11.11]{barthel2023descenttensortriangulargeometry} that $\textup{Thick}_{\otimes}(M) \mapsto \textup{Thick}_{\otimes}(M{\downarrow}_E^F)$ is a homomorphism of distributive lattices. Furthermore, conjugation also induces homomorphisms of distributive lattices in the obvious way. Altogether, this means we have a functor $\mathcal{A}_{\mathcal{E}}(G)^{op} \to \mathbf{DLat}$, defined on objects by $E \mapsto\textup{Thickid}^f(E^d)$. 
\par 
As shown in \cite[Lemma~3.14]{barthel2023descenttensortriangulargeometry}, the forgetful functor from the category of distributive lattices to the category of sets creates limits. We will use this implicitly in what follows. 
\par 
Consider the induced map $\textup{Thickid}^f(G^d) \to \prod\limits_{E \in \mathcal{E}(G)}\textup{Thickid}^f(E^d)$. It is not hard to check that this has image in $\lim\limits_{E \in \mathcal{A}_{\mathcal{E}}(G)}\textup{Thickid}^f(E^d)$. That is, we have a map \[\textup{Thickid}^f(G^d) \to \lim\limits_{E \in \mathcal{A}_{\mathcal{E}}(G)}\textup{Thickid}^f(E^d)\]
which is induced by restriction. We are interested in showing this is an isomorphism.
\par 
For the following, recall the notation $\xi_M$ from \eqref{ximeq}, as the morphism which arises from completing the unit of the adjunction $k \to M^*\otimes M$ to a triangle, when $M$ is dualisable.
\begin{proposition}\label{prop1}
    Suppose $G$ is an $\textup{H}_1\mathfrak{F}$ group of type $\textup{FP}_{\infty}$. Then the map $\textup{Thickid}^f(G^d) \to \lim\limits_{E \in \mathcal{A}_{\mathcal{E}}(G)} \textup{Thickid}^f(E^d)$ is injective. 
\end{proposition}
\begin{proof}
    Suppose we have $M,N$ such that $\textup{Thick}_{\otimes}(M{\downarrow}_E^G) = \textup{Thick}_{\otimes}(N{\downarrow}_E^G)$ for each elementary abelian $E \leq G$. 
    \par 
    Let $X \in \textup{Thick}_{\otimes}(M)$. Then for any finite subgroup $F \leq G$, we have that $X{\downarrow}_F^G \in \textup{Thick}_{\otimes}(N{\downarrow}_F^G)$. Using \cite[Corollary~2.11]{BalmSurj} we know that $\xi_{N{\downarrow}_F^G}^{\otimes n}\otimes X{\downarrow}_F^G = 0$ for some $n \geq 1$.
    \par
    Restriction is a tensor-triangulated functor, and so we know that $\xi^{\otimes n}_{N{\downarrow}_F^G} = (\xi^{\otimes n}_{N}){\downarrow}_F^G$. In particular, we have that $0 = \xi^{\otimes n}_{N{\downarrow}_F^G}\otimes X{\downarrow}_F^G = (\xi^{\otimes n}_{N}\otimes X){\downarrow}_F^G$. 
    \par 
    Since there are only finite many conjugacy classes of finite elementary abelian subgroups, we may take some $m > 0$ such that $ (\xi^{\otimes m}_{N}\otimes X^{\otimes m}){\downarrow}_E^G = 0$ for each finite elementary abelian subgroup $E \leq G$. It follows from \Cref{fiszero} that $ \xi^{\otimes l}_{N}\otimes X^{\otimes l} = 0$ for some $l \geq 0$. 
    \par 
    This implies, again by \cite[Corollary~2.11]{BalmSurj}, that $X^{\otimes l} \in \textup{Thick}_{\otimes}(N)$, and so $X \in \textup{Thick}_{\otimes}(N)$. This shows the inclusion $\textup{Thick}_{\otimes}(M) \subseteq \textup{Thick}_{\otimes}(N)$. The other inclusion holds by swapping the roles of $M$ and $N$ in the above argument. 
\end{proof}

\begin{proposition}\label{prop2}
Suppose $G$ is an $\textup{H}_1\mathfrak{F}$ group of type $\textup{FP}_{\infty}$. Then the natural map $\textup{Thickid}^f(G^d) \to \lim\limits_{E \in \mathcal{A}_{\mathcal{E}}(G)} \textup{Thickid}^f(E^d)$ is surjective.
\end{proposition}
\begin{proof}
    Suppose we have an element of the limit $(\textup{Thick}_{\otimes}(N_E))_{E \in \mathcal{E}(G)}$, where each $N_E$ is a finitely generated $kE$-module. We then consider the subset $\mathcal{V}\subseteq \textup{Proj}(H^*(G,k))$ given by the union over conjugacy classes of all finite elementary abelian subgroups $E \leq G$ of $\textup{res}^*_{G,E}(\textup{Supp}_E(N_E))$. Note that $H^*(E,k)$ is finitely generated as a module over $\lim\limits_{F \in \mathcal{A}_{\mathcal{E}}(G)} H^*(F,k)$ and hence by the $F$-isomorphism of \Cref{thm:henn}, we see that $H^*(E,k)$ is finitely generated as a module over $H^*(G,k)$ via the restriction map. Finite morphisms are closed and so this implies that $\textup{res}^*_{G,E}$ is a closed map. Using that there are only a finite number of conjugacy classes of finite elementary abelian subgroups, we conclude that $\mathcal{V}$ is a closed subset. 
    \par 
    From \Cref{carlsonzeta} we know that we can find a dualisable module $X$ such that $\textup{Supp}_G(X) = \mathcal{V}$. Note that by \Cref{inducedsupport} we know that $\textup{Supp}_G(\bigoplus\limits_{E \in \mathcal{E}(G)} N_E{\uparrow}_E^G) = \mathcal{V}$ and so by the classification of localising tensor ideals as in \Cref{newsix}, we know that $\locm(X) = \locm(\bigoplus\limits_{E \in \mathcal{E}(G)} N_E{\uparrow}_E^G)$. 
    \par 
    Let $F \leq G$ be an elementary abelian subgroup. 
    We know from \Cref{isol} that $\locm(X{\downarrow}_F^G) = \locm(\bigoplus\limits_{E \in \mathcal{E}(G)} N_E{\uparrow}_E^G{\downarrow}_F^G)$. 
    \par 
    We claim that $\locm(\bigoplus\limits_{E \in \mathcal{E}(G)} N_E{\uparrow}_E^G{\downarrow}_F^G)=\locm(N_F)$. The Mackey double coset formula shows that $\locm(N_F) \subseteq \locm(\bigoplus\limits_{E \in \mathcal{E}(G)} N_E{\uparrow}_E^G{\downarrow}_F^G)$ and so we are left to show the other inclusion.
    \par 
    For any $E \in \mathcal{E}(G)$, consider the module $N_E{\uparrow}_E^G{\downarrow}_F^G \cong \bigoplus\limits_{g \in [F\backslash G/E]} {}^gN_E{\downarrow}_{F^g\cap E}^E{\uparrow}_{F\cap {}^gE}^F$. Since we took an element of the limit originally, we know that \[\locm(N_E{\downarrow}_{F^g\cap E}^E) = \locm(N_{F^g\cap E})\]
    and
    \[\locm({}^gN_{F^g\cap E}) = \locm(N_{F\cap {}^gE})\]
    Furthermore, using the fact that $\locm(N_{F\cap {}^gE}) = \locm(N_F{\downarrow}_{F\cap {}^gE}^F)$ we see that 
    \[\locm(N_{F\cap{}^gE}{\uparrow}_{F\cap{}^gE}^F) = \locm(N_F{\downarrow}_{F\cap{}^gE}^F{\uparrow}_{F\cap{}^gE}^F)\]
    From this, we conclude that $\locm({}^gN_E{\downarrow}_{F^g\cap E}^E{\uparrow}_{F\cap{}^gE}^F) = \locm(N_F{\downarrow}_{F\cap{}^gE}^F{\uparrow}_{F\cap{}^gE}^F) \subseteq \locm(N_F)$. 
    \par 
    Altogether, this shows that $\bigoplus\limits_{g\in[F\backslash G/E]} {}^gN_E{\downarrow}_{F^g\cap E}^E{\uparrow}_{F\cap {}^gE}^F \in \locm(N_F)$ as desired. This is true for all $E \in \mathcal{E}(G)$ and so we have shown that $\locm(\bigoplus\limits_{E \in \mathcal{E}(G)} N_E{\uparrow}_E^G{\downarrow}_F^G)=\locm(N_F)$. 
    \par Summarising, we have shown that we have a dualisable object $X$ such that $\locm(X{\downarrow}_F^G) = \locm(N_F)$ for each elementary abelian subgroup $F \leq G$. We know that $X{\downarrow}_F^G$ and $N_F$ are all dualisable and hence compact objects, and so we may apply the Neeman-Thomason localisation theorem as in \cite[Corollary~2.1]{HallRydh} to see that $\textup{Thick}_{\otimes}(X{\downarrow}_F^G)=\textup{Thick}_{\otimes}(N_F)$, from which surjectivity follows.
\end{proof}
Stone duality now allows us to determine the Balmer spectrum immediately. 
\begin{theorem}\label{BALMER}
    The Balmer spectrum of $\tacstab(kG)^d$ is homeomorphic to $\underset{{E \in \mathcal{A}_{\mathcal{E}}(G)}}{\textup{colim}}\textup{Proj}(H^*(E,k))$, where the colimit is taken in the category of spectral spaces.
\end{theorem}
\begin{proof}
    We have shown in \Cref{prop1} and \Cref{prop2} that the map $\textup{Thickid}^f(G^d) \to \lim\limits_{E \in \mathcal{A}_{\mathcal{E}}(G)} \textup{Thickid}^f(E^d)$ is an isomorphism of distributive lattices. Stone duality as in \Cref{Stonedual} now allows us to conclude. 
\end{proof}
\begin{corollary}\label{bspec2}
    The Balmer spectrum of $\tacstab(kG)^d$ is homeomorphic to $\underset{{E \in \mathcal{A}_{\mathcal{E}}(G)}}{\textup{colim}}\textup{Proj}(H^*(E,k))$, where the colimit is taken in the category of topological spaces. 
\end{corollary}
\begin{proof}
    We first note that since there are only finitely many conjugacy classes of elementary abelian subgroups by \Cref{thm:henn}, the category $\mathcal{A}_{\mathcal{E}}(G)$ is equivalent to a finite category. The spectral colimit is therefore a coequalizer of the form $\begin{tikzcd}
        \bigsqcup\textup{Proj}(H^*(E,k)) \arrow[r,shift left=.75ex]\arrow[r,shift right=.75ex]& \bigsqcup\textup{Proj}(H^*(E,k)),\end{tikzcd}$ where each spectral coproduct is finite. By \cite[Theorem~7]{Hochster}, we know that finite spectral coproducts may be calculated as the topological coproduct. Since $\textup{Proj}(H^*(E,k))$ is noetherian for a finite elementary abelian $p$-group $E$, we conclude that each of these coproducts are also noetherian. 
    \par 
    It is shown in \cite[Proposition~4.12]{barthel2023descenttensortriangulargeometry} that under these conditions, the map from the topological coequalizer to the spectral coequalizer (or its image under the forgetful functor) is a homeomorphism if and only if the topological coequalizer is a $T_0$ space. However, the topological coequalizer is homeomorphic to $\textup{Proj}(H^*(G,k))$ and hence is $T_0$. From this, it follows that the spectral colimit is homeomorphic to the topological colimit, and the result is now immediate from \Cref{BALMER}. 
\end{proof}
\subsection{The telescope conjecture for $\mathrm{H}_1\mathfrak{F}$ groups of type $\mathrm{FP}_{\infty}$}
We can now show that the telescope conjecture does hold for $\mathrm{H}_1\mathfrak{F}$ groups of type $\mathrm{FP}_{\infty}$. 
\begin{proposition}\label{rigidgenerate}
    Suppose $G$ is an $\textup{H}_1\mathfrak{F}$ group of type $\textup{FP}_{\infty}$. Let $\mathcal{L}$ be a localising tensor ideal generated by modules of the form $M{\uparrow}_E^G$ for $M$ a finitely generated $kE$-module and $E \leq G$ a finite elementary abelian subgroup. Then $\mathcal{L}$ is generated as a localising tensor ideal by a set of dualisable objects. 
\end{proposition}
\begin{proof}
     We see from \Cref{inducedsupport} that $\mathcal{L}$ must have specialisation closed support. Therefore, it suffices to show that every closed subset of $\textup{Proj}(H^*(G,k))$ is the support of a dualisable object, which we have already shown in \Cref{carlsonzeta}. 
\end{proof}
\begin{corollary}\label{tel1}
Let $G$ be an $\textup{H}_1\mathfrak{F}$ group of type $\textup{FP}_{\infty}$. Then every smashing ideal is generated by a set of dualisable objects.
\end{corollary}
\begin{proof}
    Let $\mathcal{S}$ be smashing. From \Cref{smagenco} we know that $\mathcal{S}$ is generated by objects of the form $M{\uparrow}_E^G$, where $E\leq G$ is a finite elementary abelian subgroup and $M$ is a finitely generated $kE$-module. Then from \Cref{rigidgenerate} we see that in fact $\mathcal{S}$ is generated by a set of dualisable objects as desired. 
\end{proof}
\begin{theorem}\label{telcon}
    Let $G$ be an $\textup{H}_1\mathfrak{F}$ group of type $\textup{FP}_{\infty}$. Then a localising tensor ideal of $\tacstab(kG)$ is smashing if and only if it is generated by a set of dualisable objects.
\end{theorem}
\begin{proof}
    This follows immediately from \Cref{tel2} and \Cref{tel1}.
\end{proof}
\begin{corollary}\label{TELHERE}
    Let $G$ be an $\textup{H}_1\mathfrak{F}$ group of type $\textup{FP}_{\infty}$. Then the natural map $\textup{Thickid}(G^d) \to \textup{Smashid}(G)$ is bijective and so the telescope conjecture holds for $G$. 
\end{corollary}
\begin{proof}
    It follows from \Cref{telcon} that the map is surjective. Suppose that $\mathcal{T}_1,\mathcal{T}_2$ are thick ideals of $\tacstab(kG)^d$ such that $\locm(\mathcal{T}_1) = \locm(\mathcal{T}_2)$. From \Cref{newsix} we know that $\textup{Supp}_G(\mathcal{T}_1) = \textup{Supp}_G(\mathcal{T}_2)$. Now, the computation of the Balmer spectrum in \Cref{bspec2} implies that $\mathcal{T}_1 = \mathcal{T}_2$, which shows injectivity, and hence bijectivity, of the map in the statement.
\end{proof}
\bibliographystyle{amsplain}
\bibliography{specref} 

\end{document}